\documentclass[10pt]{article}  
\usepackage{amsmath}
\usepackage{amssymb}
\usepackage{graphicx}
\usepackage{amsthm}       \newtheorem{thm}{Theorem}[section]
\newtheorem{prop}[thm]{Proposition}
\newtheorem{lem}[thm]{Lemma}
\newtheorem{cor}[thm]{Corollary}  \theoremstyle{definition}
\newtheorem{df}[thm]{Definition}   \theoremstyle{definition}
\newtheorem{ques}[thm]{Question}
\newtheorem{prob}[thm]{Problem}
\newtheorem{rem}[thm]{Remark}                \theoremstyle{plain}
 \theoremstyle{definition}
\newtheorem{ex}[thm]{Example}   \def\CC{\Bbb{C}}
  
\def\CCI{\hat{\CC}}        \def\NN{\Bbb{N}} 

\def\g{\gamma}
\def\G{\Gamma}
\def\GN{\Gamma ^{\NN }}
\def\B1{{\rm\kern.32em\vrule    width.12em       height1.4ex
depth-.05ex\kern-.28em 1}}
\def\pc{\pi _{\CCI }}

\def\GNCR{\GN \times \CCI \rightarrow \GN \times \CCI }
\begin{document}
\title
{Dynamics of 
postcritically bounded  
polynomial 
semigroups III: 
classification of semi-hyperbolic semigroups and 
random Julia sets which are Jordan curves but not quasicircles
\footnote{Published in Ergodic Theory  Dynam. Systems. (2010), {\bf 30}, No. 6, 1869--1902. 
2000 Mathematics Subject Classification. 
37F10, 30D05. Keywords: Complex dynamics, polynomial semigroup, rational semigroup,  
Random complex dynamics, 
Julia set.}}
\author{Hiroki Sumi\\ 
 Department of Mathematics,    
Graduate School of Science\\ 
Osaka University \\ 1-1, \ Machikaneyama,\ Toyonaka,\ Osaka,\ 560-0043,\ 
Japan\\ E-mail: sumi@math.sci.osaka-u.ac.jp\\ 
http://www.math.sci.osaka-u.ac.jp/$\sim $sumi/}
\date{December 8, 2009}
\maketitle 
\begin{abstract}
We investigate the dynamics of polynomial semigroups (semigroups 
generated by a family of polynomial maps on the Riemann sphere $\CCI $) 
and the random dynamics of polynomials on the Riemann sphere. 
Combining the dynamics of semigroups and the fiberwise (random) dynamics, 
we give a classification of polynomial semigroups 
$G$ such that $G$ is generated by a compact family $\Gamma $, 
 the planar postcritical set of $G$ is bounded, and  
$G$ is (semi-) hyperbolic.  
In one of the classes, we have that for almost 
every sequence $\gamma \in \Gamma ^{\NN }$, 
the Julia set $J_{\gamma }$ of $\gamma $ is a Jordan curve but not a quasicircle, 
the unbounded component of $\CCI \setminus J_{\gamma }$ is a John domain, 
and the bounded component of $\CC \setminus J_{\gamma }$ is not a John domain. 
Note that this phenomenon does not hold in the usual iteration of a single 
polynomial. Moreover, we consider the dynamics of polynomial semigroups $G$ 
such that the planar postcritical set of $G$ is bounded and the Julia set is disconnected.  
Those phenomena of polynomial semigroups and random dynamics of 
polynomials that do not occur in the usual dynamics of polynomials  
are systematically investigated. 
\end{abstract} 
\section{Introduction}
\label{Intro}
This is the third paper in which the dynamics of semigroups of polynomial maps with 
bounded planar postcritical set in $\CC $ are investigated. 
This paper is self-contained and the proofs of the results of this paper are  
independent from the results in \cite{SdpbpI, S11}. 
 
The theory of complex dynamical systems, which has 
 its origin in the important 
 work of Fatou and Julia in the 1910s, 
 has been investigated by many people and discussed in depth.  
In particular, since D. Sullivan showed the famous 
``no wandering domain theorem'' using 
Teichm\"{u}ller theory 
in the 1980s, 
this subject has 
attracted 
many researchers 
from a
wide area. 
For a general 
reference on complex dynamical systems, 
see Milnor's textbook \cite{M}.   
 
 There are several 
areas
 in which we deal with 
  generalized notions of 
classical iteration theory of rational functions.   
One of them is the theory of 
dynamics of rational semigroups 
  (semigroups generated by a family of holomorphic maps on the 
  Riemann sphere $\CCI $), and another one is 
 the theory of   
 random dynamics of holomorphic maps on the Riemann sphere. 

In this paper, we will discuss 
these subjects.
 A {\bf rational semigroup} is a semigroup 
generated by a family of non-constant rational maps on 
$\CCI $, where $\CCI $ denotes the Riemann sphere,
 with the semigroup operation being  
functional composition (\cite{HM1}). A 
{\bf polynomial semigroup} is a 
semigroup generated by a family of non-constant 
polynomial maps.
Research on the dynamics of
rational semigroups was initiated by
A. Hinkkanen and G. J. Martin (\cite{HM1,HM2}),
who were interested in the role of the
dynamics of polynomial semigroups while studying
various one-complex-dimensional
moduli spaces for discrete groups,
and
by F. Ren's group(\cite{ZR, GR}), 
 who studied 
such semigroups from the perspective of random dynamical systems.
Moreover, the research 
on 
rational semigroups is related to 
that 
on 
``iterated function systems" in 
fractal geometry.  
In fact, 
the Julia set of a rational semigroup generated by a 
compact family has 
`` backward self-similarity" 
(cf. Lemma~\ref{hmslem}-\ref{bss}). 
For 
other 
research 
on rational semigroups, see 
\cite{Sta1, Sta2, Sta3, SY, SSS, 
SS, SU1, SU2, SU3}, and \cite{S1}--\cite{S8}. 

 The research 
on the  
dynamics of rational semigroups is also directly related to 
that 
on the 
random dynamics of holomorphic maps. 
The first 
study 
in this 
direction was 
by Fornaess and Sibony (\cite{FS}), and 
much research has followed. 
(See \cite{Br, Bu1, Bu2, 
BBR,GQL, S9, S8}.)   

 We remark that the complex dynamical systems 
 can be used to describe some mathematical models. For 
 example, the behavior of the population 
 of a certain species can be described as the 
 dynamical system of a polynomial 
 $f(z)= az(1-z)$ 
 such that $f$ preserves the unit interval and 
 the postcritical set in the plane is bounded 
 (cf. \cite{D}). 
It should also be remarked that 
 according to the change of the natural environment, 
 some species have several strategies to survive in the nature. 
From this point of view, 
 it is very important to consider the random 
 dynamics of such polynomials (see also Example~\ref{realpcbex}). 
For the random dynamics of polynomials on the unit interval, 
see \cite{Steins}. 
 
 We shall give some definitions 
for the 
dynamics of rational semigroups.  
\begin{df}[\cite{HM1,GR}] 
Let $G$ be a rational semigroup. We set
\[ F(G) = \{ z\in \CCI \mid G \mbox{ is normal in a neighborhood of  $z$} \} ,
 \mbox{ and }\ J(G)  = \CCI \setminus  F(G) .\] \(  F(G)\) is  called the
{\bf Fatou set}  of  $G$ and \( J(G)\)  is  called the {\bf 
Julia set} of $G$. 
We 
let 
$\langle h_{1},h_{2},\ldots \rangle $ 
denote 
the 
rational semigroup generated by the family $\{ h_{i}\} .$
More generally, for a family $\Gamma $ of non-constant rational maps, 
we denote by $\langle \Gamma  \rangle $ the rational semigroup generated by 
$\Gamma .$ 
The Julia set of the semigroup generated by 
a single map $g$ is denoted by 
$J(g).$ Similarly, we set $F(g):= F(\langle g\rangle ).$ 
\end{df}

\begin{df}\ 
\begin{enumerate}
\item 
For each rational map $g:\CCI \rightarrow \CCI $, 
we set 
$CV(g):= \{ \mbox{all critical values of }
g: \CCI \rightarrow \CCI \} .$ 
Moreover, for each polynomial map $g:\CCI \rightarrow \CCI $, 
we set $CV^{\ast }(g):= CV(g)\setminus \{ \infty \} .$ 
\item 
Let $G$ be a rational semigroup.
We set 
$$ P(G):=
\overline{\bigcup _{g\in G} CV(g)} \ (\subset \CCI ). 
$$ 
This is called the {\bf postcritical set} of $G.$
Furthermore, for a polynomial semigroup $G$,\ we set 
$P^{\ast }(G):= P(G)\setminus \{ \infty \} .$ This is 
called the {\bf planar postcritical set}
(or {\bf finite postcritical set}) 
 of $G.$
We say that a polynomial semigroup $G$ is 
{\bf postcritically bounded} if 
$P^{\ast }(G)$ is bounded in $\CC .$ 
\end{enumerate}
\end{df}
\begin{rem}
\label{pcbrem}
Let $G$ be a rational semigroup 
generated by a family $\Lambda $ of rational maps. 
Then, we have that 
$P(G)=\overline{\bigcup _{g\in G\cup \{ Id\} }\ g(\bigcup _{h\in \Lambda }CV(h))}$, 
where Id denotes the identity map on $\CCI .$   
Thus $g(P(G))\subset P(G)$ for each $g\in G.$  
From this formula, one can figure out how the set 
$P(G)$ (resp. $P^{\ast }(G)$) spreads in $\CCI $ (resp. $\CC $). 
In fact, in Section~\ref{Const}, using the above formula, 
we present a way to construct examples of postcritically bounded 
polynomial semigroups (with some additional properties). Moreover, 
from the above formula, one may, in the finitely generated case, 
use a computer to see if a polynomial semigroup $G$ is postcritically bounded much in the same way 
as one verifies the boundedness of the critical orbit for the maps $f_{c}(z)=z^{2}+c.$   
\end{rem}
\begin{ex}
\label{realpcbex}
Let 
$\Lambda := \{ h(z)=cz^{a}(1-z)^{b}\mid 
a,b\in \NN  ,\ c>0,\  
c(\frac{a}{a+b})^{a}(\frac{b}{a+b})^{b}$ $\leq 1\} $ 
and let $G$ be the polynomial semigroup generated by 
$\Lambda .$ 
Since for each $h\in \Lambda $, 
$h([0,1])\subset [0,1]$ and 
$CV^{\ast }(h)\subset [0,1]$, 
it follows that each subsemigroup $H$ of $G$ is postcritically 
bounded. 
\end{ex}
\begin{rem}
\label{pcbound}
It is well-known that for a polynomial $g$ with 
$\deg (g)\geq 2$, 
$P^{\ast }(\langle g\rangle )$ is bounded in $\CC $ if and only if 
$J(g)$ is connected (\cite[Theorem 9.5]{M}).
\end{rem}
As mentioned in Remark~\ref{pcbound}, 
 the planar postcritical set is one 
piece of important information 
regarding the dynamics of polynomials. 
Concerning 
the theory of iteration of quadratic polynomials, 
we have been investigating the famous ``Mandelbrot set''.
    
When investigating the dynamics of polynomial semigroups, 
it is natural for us to 
discuss the relationship between 
  the planar postcritical set and the 
  figure of the Julia set.
The first question in this 
regard 
is: 
\begin{ques}
Let $G$ be a polynomial semigroup such that each 
element $g\in G$ is of degree two or more.
Is $J(G)$ necessarily connected when $P^{\ast }(G)$ is 
bounded in $\CC $?
\end{ques}
The answer is {\bf NO.}
\begin{ex}[\cite{SY}]
Let $G=\langle z^{3}, \frac{z^{2}}{4}\rangle .$ 
Then $P^{\ast }(G) =\{ 0\} $ 
(which is bounded in $\CC $)
and $J(G)$ is disconnected ($J(G)$ is a Cantor set 
of round circles). Furthermore,\ 
according to 
\cite[Theorem 2.4.1]{S5},  
it can be shown that 
a small					 
perturbation $H$ of $G$ 
 still satisfies that 
 $P^{\ast }(H) $ is 
 bounded in $\CC $  and that $J(H)$ is disconnected. 
 ($J(H)$ is a 
Cantor set of quasi-circles with uniform dilatation.)
\end{ex}
\begin{ques}
What happens if $P^{\ast }(G) $ is bounded in $\CC $ 
and $J(G)$ is disconnected? 
\end{ques}
\begin{prob}
Classify postcritically bounded polynomial semigroups.
\end{prob}
\begin{df} 
Let ${\cal G} $ be the set of all polynomial semigroups 
$G$ with the following 
properties:
\begin{itemize}
\item 
 each element of $G$ is of degree 
two or more, and  
\item  $P^{\ast }(G)$ is 
bounded in $\CC $, i.e., $G$ is postcritically bounded.
\end{itemize}   
Furthermore, we set 
${\cal G}_{con}=
\{ G\in {\cal G}\mid 
J(G)\mbox{ is connected}\} $ and 
${\cal G}_{dis}=
\{ G\in {\cal G}\mid 
J(G)\mbox{ is disconnected}\}.$ 
\end{df}  

We also investigate the dynamics of hyperbolic or semi-hyperbolic 
polynomial semigroups. 
\begin{df}
Let $G$ be a rational semigroup. 
\begin{enumerate}
\item 
We say that 
$G$ is hyperbolic if $P(G)\subset F(G).$ 
\item We say that $G$ is semi-hyperbolic if 
there exists a number $\delta >0$ and a 
number $N\in \NN $ such that,  
for each $y\in J(G)$ and each $g\in G$, 
we have $\deg (g:V\rightarrow B(y,\delta ))\leq N$ for 
each connected component $V$ of $g^{-1}(B(y,\delta ))$, 
where $B(y,\delta )$ denotes the ball of radius $\delta $ 
with center $y$ with respect to the spherical distance, 
and $\deg (g:\cdot \rightarrow \cdot )$ denotes the 
degree of finite branched covering. 
(For the background of semi-hyperbolicity, see \cite{S1} and \cite{S4}.) 
\end{enumerate}
\end{df} 
\begin{rem}
There are many nice properties of hyperbolic or semi-hyperbolic 
rational semigroups. For example, for a finitely generated 
semi-hyperbolic rational semigroup $G$ , there exists an attractor in the Fatou set (\cite{S1, S4}),  
and the Hausdorff dimension $\dim _{H}(J(G))$ of the Julia set is less than or equal to the 
critical exponent $s(G)$ of the Poincar\'{e} series of $G$ (\cite{S1}).  
If we assume further the ``open set condition'', 
then $\dim _{H}(J(G))=s(G)$ (\cite{S6, SU3}). 
Moreover, if $G\in {\cal G}$ is generated by a compact set $\Gamma $ and if $G$ is semi-hyperbolic, 
then for each sequence $\gamma \in \Gamma ^{\NN }$, 
the basin of infinity for $\gamma $ is a John domain and the Julia set of $\gamma $ is connected and locally connected 
(\cite{S4}). This fact will be used in the proofs of the main results of this paper.       
\end{rem}

In this paper, we 
classify the 
semi-hyperbolic, postcritically bounded, 
polynomial semigroups generated by a compact family $\G $ of 
polynomials.  
We show that such a semigroup $G$ satisfies either  
(I) every fiberwise Julia 
set is a quasicircle 
with uniform distortion, 
or (II) for almost 
every sequence $\g \in \GN $, the Julia set $J_{\g }$ is a 
Jordan curve but not a quasicircle, the basin of infinity 
$A_{\g }$ is a John domain, and the bounded component $U_{\g }$ of the Fatou 
set is not a John domain, or (III) for every $\alpha ,\beta \in \GN $, 
the intersection of the Julia sets $J_{\alpha }$ and $J_{\beta }$ is not empty, 
and $J(G)$ is arcwise connected (cf. Theorem~\ref{mainthran1}). 
Furthermore, we also 
classify the 
hyperbolic, postcritically bounded, 
polynomial semigroups generated by a compact family $\G $ of 
polynomials.  
We show that such a semigroup $G$ satisfies either 
(I) above, or (II) above, or (III)': for every $\alpha ,\beta \in \GN $, 
the intersection of the Julia sets $J_{\alpha }$ and $J_{\beta }$ is not empty, $J(G)$ is arcwise connected, and for every sequence $\g \in \GN $, 
there exist infinitely many bounded components of $F_{\g }$ 
(cf. Theorem~\ref{mainthran2}). We give some examples 
of 
situation (II) above 
(cf. 
Example~\ref{jbnqexfirst}, figure~\ref{fig:dcgraph},  
Example~\ref{jbnqex},  and 
Section~\ref{Const}). 
  Note that 
situation (II) above is 
a 
special phenomenon of 
  random dynamics of polynomials that  
  does not occur in the usual dynamics of polynomials. 

The key to 
investigating the 
dynamics of 
postcritically bounded polynomial semigroups is 
the density of repelling fixed points in the Julia set (cf. 
Theorem~\ref{repdense}), which 
can be shown by 
an application of 
the Ahlfors five island theorem, and the lower semi-continuity of 
$\g \mapsto J_{\g }$ (Lemma~\ref{fibfundlem}-\ref{fibfundlem2}), which is a consequence of potential theory.   
 The key to 
investigating the 
dynamics of semi-hyperbolic polynomial semigroups is, 
the continuity of the map $\g \mapsto J_{\g }$ 
(this is highly nontrivial; see \cite{S1}) 
and the 
Johnness of the basin $A_{\g }$ of infinity (cf. \cite{S4}). 
Note that the continuity of the map $\g \mapsto J_{\g }$ 
does not hold in general, if we do not assume semi-hyperbolicity. 
 Moreover, one of the 
original aspects 
of this paper 
is the idea of 
 ``combining both 
the theory of rational semigroups and 
 that of random complex dynamics". It is quite natural to 
investigate both fields simultaneously. However, 
no study 
thus far has done so.

 Furthermore, in Section~\ref{Const},  
we 
provide 
a way of constructing examples of 
 postcritically bounded polynomial semigroups 
 with 
some additional properties (disconnectedness of Julia set, 
semi-hyperbolicity, hyperbolicity, etc.) 
(cf. Lemma~\ref{Constprop}, \ref{shshfinprop}, \ref{l:omegahopen}, \ref{l:twohypp}, \ref{l:disjfinv}).  
By using this, we will see how easily situation (II) above occurs, 
and we obtain many examples of 
situation (II) above.

As wee see in Example~\ref{realpcbex} and Section~\ref{Const}, 
it is not difficult to construct many examples,  
it is not difficult to verify the hypothesis ``postcritically 
bounded'', 
and the class of postcritically bounded polynomial semigroups is 
very wide.  

 Throughout the paper, we will see some phenomena in polynomial 
 semigroups or random dynamics of polynomials that do not occur in 
 the usual dynamics of polynomials. Moreover, those phenomena and their mechanisms are   
 systematically investigated.
 
 In Section~\ref{Main}, we present the main results 
 of this paper. We give some tools in Section~\ref{Tools}. 
 The proofs of the main results are given in Section~\ref{Proofs}. 
In Section~\ref{Const}, we present many examples.  
\ 

There are many applications of the results of postcritically 
bounded polynomial semigroups in many directions. 
 In subsequent papers \cite{S8,Snew}, we 
 will investigate 
Markov process on $\CCI $ associated with 
the random dynamics of polynomials and 
we will consider the probability $T_{\infty }(z)$ 
of tending to $\infty \in \CCI $ 
starting with the initial value $z\in \CCI .$ 
It will be shown in \cite{S8,Snew} that 
if the associated polynomial semigroup $G$ 
is postcritically bounded and the Julia set is 
disconnected, then the function $T_{\infty }$ defined on $\CCI $ 
has many interesting properties which are 
similar to those of the Cantor function. 
For example, under certain conditions, 
$T_{\infty }$ is continuous on $\CCI$, varies precisely on $J(G)$ which is a thin 
fractal set, 
and $T_{\infty }$ has a kind of monotonicity.  
Such a kind of ``singular functions on the 
complex plane'' appear very naturally in 
random dynamics of polynomials 
and the 
study of the dynamics of postcritically polynomial semigroups are the keys to 
investigating that. 
(The above results have been announced in \cite{S9, S10}.) 

 Moreover, as illustrated before, 
 it is very important for us to recall that 
 the complex dynamics can be applied to describe some mathematical models. 
 For example, the behavior of the population of a 
 certain species can be described as the dynamical systems
  of a polynomial $h$ such that $h$ preserves the unit interval  
 and the postcritical set in the plane is bounded. 
 When one considers such a model, 
 it is very natural to consider the random dynamics of 
 polynomials with bounded postcritical set in the plane 
 (see Example~\ref{realpcbex}).    

In \cite{SdpbpI}, we 
investigate the dynamics of postcritically bounded polynomial semigroups $G$ 
which is possibly generated by a non-compact family. The structure of the Julia set is 
deeply studied, and for such a $G$ with disconnected Julia set, it is shown that $J(G)\subset \CC $, 
and that if $A$ and $B$ are two connected components of $J(G)$, 
then one of them surrounds the other. Therefore the space ${\cal J}_{G}$ of all connected components of 
$J(G)$ has an intrinsic total order. 
Moreover, we show that for each $n\in \NN \cup \{ \aleph _{0}\} $, 
there exists a finitely generated postcritically bounded polynomial semigroup $G$ 
such that the cardinality of the space of all connected components of $J(G)$ is equal to $n.$ 
In \cite{S11}, by using the results in \cite{SdpbpI}, 
we investigate the fiberwise (random) dynamics of polynomials which are associated with 
a postcritically bounded polynomial semigroup $G.$ We will present some sufficient conditions 
for a fiberwise Julia set to be a Jordan curve but not a quasicircle. 
Moreover, we will investigate the limit functions of the fiberwise dynamics.     
 In the subsequent paper \cite{SS}, we will give some 
further results on postcritically 
 bounded polynomial semigroups, based on \cite{SdpbpI} and this paper. 
 Moreover, in the subsequent paper \cite{S7}, 
 we will define a new kind of cohomology theory, in order to 
 investigate the action of finitely generated semigroups, and 
we will apply it to the study of 
the dynamics of postcritically bounded polynomial semigroups.  

\ 

\noindent {\bf Acknowledgement:} 
The author 
thanks 
R. Stankewitz for many 
valuable comments.  
\section{Main results}
\label{Main}
In this section we present the statements of the main results.
The proofs are given in Section~\ref{Proofs}.
In order to present the main results, 
we need some notations and definitions. 
\begin{df}
We set 
Rat : $=\{ h:\CCI \rightarrow \CCI \mid 
h \mbox { is a non-constant rational map}\} $
endowed with the 
distance $\eta $ which is defined by $\eta (h_{1},h_{2}):=\sup _{z\in \CCI }d(h_{1}(z),h_{2}(z))$, 
where $d$ denotes the spherical distance on $\CCI .$ 
We set 
Poly :$=\{ h:\CCI \rightarrow \CCI 
\mid h \mbox{ is a non-constant polynomial}\} $ endowed with 
the relative topology from Rat.   
Moreover, we set 
Poly$_{\deg \geq 2}
:= \{ g\in \mbox{Poly}\mid \deg (g)\geq 2\} $ 
endowed with the relative topology from 
Rat.  
\end{df}
\begin{rem}
Let $d\geq 1$,  $\{ p_{n}\} _{n\in \NN }$ a 
sequence of polynomials of degree $d$, 
and $p$ a polynomial.  
Then, $p_{n}\rightarrow p$ in Poly if and only if 
the coefficients converge appropriately and $p$ is of degree $d.$ 
\end{rem}
\begin{df}
 For a polynomial semigroup $G$,\ we set 
$$ \hat{K}(G):=\{ z\in \CC 
\mid \bigcup _{g\in G}\{ g(z)\} \mbox{ is bounded in }\CC \} $$ 
and call $\hat{K}(G)$ the {\bf smallest filled-in Julia set} of 
$G.$ 
For a polynomial $g$, we set $K(g):= \hat{K}(\langle g\rangle ).$ 
\end{df}
\begin{df}
For a set $A\subset \CCI $, we denote by int$(A)$ the set of 
all  
interior points of $A.$ 
\end{df} 

\begin{df}[\cite{S1,S4}]
\ 
\begin{enumerate}
\item  Let $X$ be a compact metric space, 
$g:X\rightarrow X$ a continuous 
map, and $f:X\times \CCI \rightarrow 
X\times \CCI $ a continuous map. 
We say that $f$ is a rational skew 
product (or fibered rational map on 
trivial bundle $X\times \CCI $) 
over $g:X\rightarrow X$, if 
$\pi \circ f=g\circ \pi $ where 
$\pi :X\times \CCI \rightarrow X$ denotes 
the canonical projection,\ and 
if, for each $x\in X$, the 
restriction 
$f_{x}:= f|_{\pi ^{-1}(\{ x\} )}:\pi ^{-1}(\{ x\}) \rightarrow 
\pi ^{-1}(\{ g(x)\} )$ of $f$ is a 
non-constant rational map,\ 
under the canonical identification 
$\pi ^{-1}(\{ x'\} )\cong \CCI $ for each 
$x'\in X.$ Let $d(x)=\deg (f_{x})$, for each 
$x\in X.$
Let $f_{x,n}$ be 
the rational map 
defined by: $f_{x,n}(y)=\pi _{\CCI }(f^{n}(x,y))$, 
for each $n\in \NN ,x\in X$ and $y\in \CCI $,   
where    
$\pi _{\CCI }:X\times \CCI 
\rightarrow  \CCI $ is the projection map.

  Moreover, if $f_{x,1}$ is a polynomial for each $x\in X$, 
 then we say that $f:X\times \CCI \rightarrow X\times \CCI $ is a 
 polynomial skew product over $g:X\rightarrow X.$  
\item  
Let $\G $ be a compact subset of Rat. 
We set  $\GN := \{ \g =(\g _{1}, \g_{2},\ldots )\mid \forall j,\g _{j}\in \G\} $ endowed with the product topology. This is a compact metric space.
Let  
$\sigma :\GN \rightarrow \GN $ be the shift map, which is defined by  
$\sigma (\g _{1},\g _{2},\ldots ):=(\g _{2},\g _{3},\ldots )
.$ Moreover,\ 
we define a map
$f:\GN \times \CCI \rightarrow 
\GN \times \CCI $ by:
$(\g ,y) \mapsto (\sigma (\g ),\g _{1}(y)),\ $
where $\g =(\g_{1},\g_{2},\ldots ).$
This is called 
{\bf the skew product associated with 
the family $\G $ of rational maps.} 
Note that $f_{\g ,n}(y)=\g _{n}\circ \cdots \circ \g _{1}(y).$ 
\end{enumerate}

\end{df}
\begin{rem}
Let $f:X\times \CCI \rightarrow X\times \CCI  $ 
be a rational skew product over 
$g:X\rightarrow X.$ Then, the function 
$x\mapsto d(x)$ is continuous on $X.$ 
\end{rem}
\begin{df}[\cite{S1, S4}]
Let $f:X\times \CCI 
\rightarrow X\times \CCI $ be a 
rational skew product over 
$g:X\rightarrow X.$ Then, we use the following notation. 
\begin{enumerate}
\item For each $x\in X$ and $n\in \NN $, we set 
$f_{x}^{n}:=
f^{n}|_{\pi ^{-1}(\{ x\} )}:\pi ^{-1}(\{ x\} )\rightarrow 
\pi ^{-1}(\{ g ^{n}(x)\} )\subset X\times \CCI .$
\item For each $x\in X$,  
we denote by $F_{x}(f)$ 
the set of points 
$y\in \CCI $ which have a neighborhood $U$ 
in $\CCI $ such that 
$\{ f_{x,n}:U\rightarrow 
\CCI \} _{n\in \NN }$
is normal. Moreover, we set 
$F^{x}(f):= \{ x\} \times F_{x}(f)\ (\subset X\times \CCI ).$  
\item For each $x\in X$, 
we set 
$J_{x}(f):=\CCI \setminus 
F_{x}(f).$ Moreover, we set 
$J^{x}(f):= \{ x\} \times J_{x}(f)$ $ (\subset X\times \CCI ).$ 
These sets $J^{x}(f)$ and $J_{x}(f)$ are called the 
fiberwise Julia sets.
\item We set 
$\tilde{J}(f):=
\overline {\bigcup _{x\in X}J^{x}
(f)}$, where the closure is taken in the product space $X\times \CCI .$
\item For each $x\in X$, we set 
$\hat{J}^{x}(f):=\pi ^{-1}(\{ x\} )\cap \tilde{J}(f).$ 
Moreover, we set $\hat{J}_{x}(f):= 
\pc (\hat{J}^{x}(f)).$ 
\item We set $\tilde{F}(f):=(X\times \CCI)\setminus 
\tilde{J}(f).$
\end{enumerate}

\end{df}
\begin{rem}
We have $\hat{J}^{x}(f)\supset J^{x}(f)$ and 
$\hat{J}_{x}(f)\supset J_{x}(f).$ 
However, 
strict containment can occur. 
For example, let $h_{1}$ be a polynomial having a Siegel disk 
with center $z_{1}\in \CC .$ 
Let $h_{2}$ be a polynomial such that 
$z_{1}$ is a repelling fixed point of $h_{2}.$ 
Let $\G =\{ h_{1},h_{2}\} .$  
Let $f:\G \times \CCI \rightarrow \G \times \CCI $ be 
the skew product associated with the family $\G .$ 
Let $x =(h_{1},h_{1},h_{1},\ldots )\in \GN .$ 
Then, $(x,z_{1})\in \hat{J}^{x}(f)\setminus  J^{x}(f)$ and 
$z_{1}\in \hat{J}_{x}(f)\setminus J_{x}(f).$ 
\end{rem}
\begin{df}
Let $f:X\times \CCI \rightarrow X\times \CCI $ be a 
polynomial skew product over $g:X\rightarrow X.$ 
Then for each $x\in X$, we set
$K_{x}(f):=
\{ y\in \CCI  \mid 
\{ f_{x,n}(y)\} _{n\in \NN }
\mbox{ is bounded } $ in $\CC \}  $, 
 and 
 $A_{x}(f):=\{ y\in \CCI 
\mid f_{x,n}(y)\rightarrow \infty 
,\ n\rightarrow \infty \} .$
Moreover, we set 
$K^{x}(f):= \{ x\} \times K_{x}(f) \ (\subset 
X\times \CCI ) $ and 
$A^{x}(f):= \{ x\} \times A_{x}(f)\ (\subset X\times \CCI ).$ 
\end{df}
\begin{df}
Let $f:X\times \CCI 
\rightarrow X\times \CCI $ be a 
rational skew product over 
$g:X\rightarrow X.$ 
We set
$$C(f):= \{ (x,y)\in X\times \CCI \mid y \mbox{ is a critical point of }
f_{x,1}\} .$$ 
Moreover, we set   
$P(f):=\overline{\bigcup _{n\in \NN  }f^{n}(C(f))}, $
where the closure is taken in the product space $X\times \CCI .$ 
This $P(f)$ is called the {\bf fiber-postcritical set} of 
$f.$ 

 We say that $f$ is hyperbolic 
(along fibers) if 
$P(f)\subset F(f).$
\end{df}
\begin{df}[\cite{S1}]
Let $f:X\times \CCI \rightarrow X\times \CCI $ be a rational skew 
product over $g:X\rightarrow X.$ Let $N\in \NN .$ 
We say that a point $(x_{0},y_{0})\in X\times \CCI $ belongs to 
$SH_{N}(f)$ if there exists a neighborhood $U$ of 
$x_{0}$ in $X$ and a positive number $\delta $ such that 
for any $x\in U$, any $n\in \NN $, any $x_{n}\in g^{-n}(x)$, 
and any connected component 
$V$ of $(f_{x_{n},n})^{-1}(B(y_{0},\delta ))$, 
$\deg (f_{x_{n},n}:V\rightarrow B(y_{0},\delta ))\leq N.$ 
Moreover, we set 
$UH(f):= (X\times \CCI )\setminus \cup _{N\in \NN }SH_{N}(f).$ 
We say that $f$ is semi-hyperbolic (along fibers) if 
$UH(f)\subset \tilde{F}(f).$ 
\end{df}
\begin{rem}
\label{hypskewsemigrrem}
Let $\G $ be a compact subset of Rat and let 
$f:\GN \times \CCI \rightarrow \GN \times \CCI $ be the skew product 
associated with $\G .$  Let $G$ be the rational semigroup generated by 
$\G . $ Then, by Lemma~\ref{fiblem}-\ref{pic}, 
it is easy to see that $f$ is semi-hyperbolic if and only if 
$G$ is semi-hyperbolic. Similarly, it is easy to see that 
$f$ is hyperbolic if and only if $G$ is hyperbolic. 
\end{rem}
\begin{df}
Let $K\geq 1.$  A Jordan curve $\xi $ in $\CCI $ is said to be 
a $K$-quasicircle, if $\xi $ is the image of $S^{1}(\subset \CC )$  
under a $K$-quasiconformal 
homeomorphism $\varphi :\CCI \rightarrow \CCI .$  
(For the definition of a quasicircle and a quasiconformal homeomorphism, see 
\cite{LV}.) 
\end{df} 
\begin{df}
Let $V$ be a subdomain of $\CCI $ such that 
$\partial V\subset \CC .$  
We say that $V$ is a John domain if there exists a 
constant $c>0$ and a point $z_{0}\in V$ ($z_{0}=\infty $ when 
$\infty \in V$) satisfying the following:  
for all $z_{1}\in V$ there exists an arc $\xi \subset V$ connecting 
$z_{1}$ to $z_{0}$ such that 
for any $z\in \xi $, we have 
$\min \{ |z-a|\mid a\in \partial V\} \geq c|z-z_{1}|.$
(Note: in this paper, if we consider a John domain $V$, we require 
that $\partial V\subset \CC .$ 
However, in the original notion of John domain, 
more general concept of John domains $V$ was given, without assuming $\partial V\subset \CC $ (\cite{NV}).)    
\end{df}
\begin{rem}
Let $V$ be a simply connected domain in $\CCI $ such that 
$\partial V\subset \CC .$ 
It is well-known that  
if $V$ is a John domain, then 
$\partial V$ is locally connected (\cite[page 26]{NV}). 
Moreover, a Jordan curve $\xi \subset \CC $ is a 
quasicircle if and only if both components of $\CCI \setminus \xi $ are 
John domains (\cite[Theorem 9.3]{NV}).  
\end{rem}

\begin{df}
Let $X$ be a complete metric space. 
A subset $A$ of $X$ is said to be residual if 
$X\setminus A$ is a countable union of nowhere dense subsets of 
$X.$ Note that by Baire Category Theorem, a residual set $A$ is dense in $X.$ 
\end{df}
\begin{df}
For any connected sets $K_{1}$ and 
$K_{2}$ in $\CC ,\ $  ``$K_{1}\leq K_{2}$'' indicates that 
$K_{1}=K_{2}$, or $K_{1}$ is included in 
a bounded component of $\CC \setminus K_{2}.$ Furthermore, 
``$K_{1}<K_{2}$'' indicates $K_{1}\leq K_{2}$ 
and $K_{1}\neq K_{2}.$ Note that 
``$ \leq $'' is a partial order in 
the space of all non-empty compact connected 
sets in $\CC .$ This ``$\leq $" is called 
the {\bf surrounding order.} 
\end{df}

 Let $\tau $ be a Borel probability measure on Poly$_{\deg \geq 2}.$  
We consider the independent and identically distributed (abbreviated by i.i.d.) random dynamics on $\CCI $ such that 
at every step we choose a polynomial map 
$h:\CCI \rightarrow \CCI $ according to the distribution $\tau .$ 
(Hence, this is a kind of Markov process on $\CCI .$ )  
\begin{df} 
For a Borel probability measure $\tau $ on  Poly$_{\deg \geq 2}$, 
we denote by $\G_{\tau }$ the topological support of $\tau $ in  
Poly$_{\deg \geq 2}.$ (Hence, $\G_{\tau }$ is a closed set in 
Poly$_{\deg \geq 2}.$) 
Moreover, we denote by $\tilde{\tau } $ the infinite product measure $\otimes _{j=1}^{\infty } \tau .$  
This is a Borel probability measure on 
$\G_{\tau }^{\NN }.$ 
 Furthermore, we denote by $G_{\tau }$ the polynomial semigroup 
generated by $\G _{\tau }.$  
\end{df} 
We present a result on compactly generated, semi-hyperbolic, 
polynomial semigroups in ${\cal G}.$ 
\begin{thm}
\label{mainthran1}
Let $\G $ be a non-empty compact subset of 
{\em Poly}$_{\deg \geq 2}.$ 
Let 
$f:\G ^{\NN }\times \CCI \rightarrow 
\G ^{\NN }\times \CCI $ be the skew product 
associated with the family $\G $ of polynomials.
Let $G$ be the polynomial semigroup generated by $\G .$ 
Suppose that $G \in {\cal G}$ and 
that $G$ is semi-hyperbolic. 
Then, exactly one of the following three statements 
\ref{mainthran1-1}, \ref{mainthran1-2}, and 
\ref{mainthran1-3} holds.
\begin{enumerate}
\item 
\label{mainthran1-1}
$G$ is hyperbolic. Moreover, there exists a constant 
$K\geq 1$ such that for each $\g \in \G ^{\NN }$, 
$J_{\g }(f)$ is a $K$-quasicircle.
\item \label{mainthran1-2}
There exists a residual Borel subset ${\cal U}$ of 
$\G^{\NN }$ such that,  
for each Borel probability measure 
$\tau $ on {\em Poly}$_{\deg \geq 2}$ with $\G _{\tau }=\G $, we have 
$\tilde{\tau }({\cal U})=1$, and such that,  
for each $\g \in {\cal U}$, 
$J_{\g }(f)$ is a Jordan curve but not a quasicircle,  
$A_{\g }(f)$ is a John domain, 
and the bounded component of $F_{\g }(f)$ is not a 
John domain.
Moreover, there exists a dense subset ${\cal V}$ of 
$\G ^{\NN }$ such that,  
for each $\g \in {\cal V}$, $J_{\g }(f)$ is not a 
Jordan curve.  
Furthermore, there exist two elements $\alpha ,\beta \in \G ^{\NN }$ such that 
$J_{\beta }(f)<J_{\alpha }(f).$ (Remark: by Lemma~\ref{fibconnlem}, 
for each $\rho \in \GN $, $J_{\rho }(f)$ is connected.)  
\item \label{mainthran1-3}
There exists a dense subset ${\cal V}$ of $\G ^{\NN }$ such that 
for each $\g \in {\cal V}$, $J_{\g }(f)$ is not a Jordan curve. 
Moreover, for each  $\alpha ,\beta \in \G ^{\NN }$, 
$J_{\alpha }(f)\cap J_{\beta }(f)
\neq \emptyset .$ Furthermore, 
$J(G)$ is arcwise connected. 
\end{enumerate}
\end{thm} 
\begin{cor}
\label{rancor2}
Let 
$\G $ be a non-empty compact subset of 
{\em Poly}$_{\deg \geq 2}.$  
Let 
$f:\G ^{\NN }\times \CCI \rightarrow 
\G ^{\NN }\times \CCI $ be the skew product 
 associated with the family $\G $ of polynomials. 
Let $G$ be the polynomial semigroup generated by 
$\G .$ 
Suppose that $G \in {\cal G}_{dis}$ and that $G$ is 
semi-hyperbolic. Then, either statement \ref{mainthran1-1} or 
statement \ref{mainthran1-2} in Theorem~\ref{mainthran1} holds. 
In particular, for any Borel Probability measure 
$\tau $ on {\em Poly}$_{\deg \geq 2}$ with $\G _{\tau }=\G $, 
for almost every $\g \in \G _{\tau }^{\NN }$ with respect to 
$\tilde{\tau }$, $J_{\g }(f)$ is a Jordan curve.  
\end{cor} 
We now classify compactly generated, hyperbolic, polynomial semigroups 
in ${\cal G}.$ 
\begin{thm}
\label{mainthran2}
Let $\G $ be a non-empty compact subset of {\em Poly}$_{\deg \geq 2}.$ 
Let $f:\GN \times \CCI \rightarrow \GN \times \CCI $ be the 
skew product associated with the family $\G .$ 
Let $G$ be the polynomial semigroup generated by $\G .$ 
Suppose that $G\in {\cal G}$ and that $G$ is hyperbolic. 
Then, exactly one of the following three statements \ref{mainthran2-1}, 
\ref{mainthran2-2}, and \ref{mainthran2-3}
 holds. 
\begin{enumerate}
\item 
\label{mainthran2-1}
There exists a constant $K\geq 1$ such that for each 
$\g \in \GN $, $J_{\g }(f)$ is a $K$-quasicircle.
\item 
\label{mainthran2-2}
There exists a residual Borel subset ${\cal U}$ of $\GN $ such that,  
for each Borel probability measure $\tau $ on {\em Poly}$_{\deg \geq 2}$ 
with $\G _{\tau }=\G $, we have $\tilde{\tau }({\cal U})=1$, 
and such that,  for each $\g \in {\cal U}$, $J_{\g }(f)$ is a 
Jordan curve but not a quasicircle, 
$A_{\g }(f)$ is a John domain, 
and the bounded component 
of $F_{\g }(f)$ is not a John domain.   
Moreover, there exists a dense subset ${\cal V}$ of 
$\GN $
 such that, for each $\g \in {\cal V}$, 
$J_{\g }(f)$ is a quasicircle.
Furthermore, there exists a dense subset 
${\cal W}$ of $\GN $ such that, for 
each $\g \in {\cal W}$, there are infinitely many bounded connected 
components of $F_{\g }(f).$  
\item 
\label{mainthran2-3}
For each $\g \in \GN $, there are infinitely many bounded 
connected components 
of $F_{\g }(f)$. Moreover, 
for each $\alpha ,\beta \in \GN $, 
$J_{\alpha }(f)\cap J_{\beta }(f)\neq 
\emptyset .$ Furthermore, 
$J(G)$ is arcwise connected.  
\end{enumerate}

\end{thm}
\begin{ex}
\label{jbnqexfirst}
Let $g_{1}(z):=z^{2}-1$ and $g_{2}(z):= \frac{z^{2}}{4}.$ 
Let $\G :=\{ g_{1}^{2}, g_{2}^{2}\} .$ 
Let $f:\GN\times \CCI \rightarrow \GN \times \CCI $ be the 
skew product associated with $\Gamma .$ 
Moreover, let $G$ be the polynomial semigroup generated by 
$\G .$ Let $D:= \{ z\in \CC \mid |z|<0.4\} .$ Then, 
it is easy to see $g_{1}^{2}(D)\cup g_{2}^{2}(D)\subset D.$ Hence, 
$D\subset F(G).$ 
Let $U$ be a small disk around $-1.$ 
Then $g_{1}^{2}(U)\subset U$ and $g_{2}^{2}(U)\subset D.$ Therefore $U\subset F(G).$ 
Moreover, by Remark~\ref{pcbrem}, we have that 
$P^{\ast }(G)=
\overline{\bigcup _{g\in G\cup \{ Id\} }g(\{ 0,-1\} )} 
\subset D\cup U\subset F(G).$ Hence, $G\in {\cal G}$ and $G$ is hyperbolic.  
Furthermore, let $K:=\{ z\in \CC \mid 0.4\leq |z|\leq 4\} .$ Then, 
it is easy to see that $(g_{1}^{2})^{-1}(K)\cup (g_{2}^{2})^{-1}(K)\subset K$ and 
$(g_{1}^{2})^{-1}(K)\cap (g_{2}^{2})^{-1}(K)=\emptyset .$ 
Combining this with Lemma~\ref{hmslem}-\ref{backmin} and Lemma~\ref{hmslem}-\ref{bss}, 
we obtain that $J(G)$ is disconnected. Therefore, 
$G\in {\cal G}_{dis}.$ 
Let $h_{i}:= g_{i}^{2}$ for each $i=1,2.$ 
Let $0<p_{1},p_{2}<1$ with $p_{1}+p_{2}=1.$ 
Let $\tau :=\sum _{i=1}^{2}p_{i}\delta _{h_{i}}.$  
Since $J(g_{1}^{2})$ is not a Jordan curve, 
from Theorem~\ref{mainthran2}, it follows that for almost every 
$\gamma \in \GN $ with respect to $\tilde{\tau }$, 
$J_{\g }(f)$ is a Jordan curve but not a quasicircle, and  
$A_{\g }(f)$ is a John domain but the bounded component 
of $F_{\g }(f)$ is not a John domain. 
(See figure~\ref{fig:dcgraph}: the Julia set of $G$.)  
In this example, for each connected component $J$ of $J(G)$, 
there exists a unique $\gamma \in \GN $ such that 
$J=J_{\gamma }(f).$ 
\begin{figure}[htbp]
\caption{The Julia set of $G=\langle g_{1}^{2},g_{2}^{2}\rangle $, 
where $g_{1}(z):= z^{2}-1, g_{2}(z):= \frac{z^{2}}{4}.$ 
For a.e.$\gamma $, $J_{\gamma }(f)$ is a Jordan curve but not a quasicircle, 
$A_{\gamma }(f)$ is a John domain, and the bounded component of 
$F_{\gamma }(f)$ is not a John domain. For each connected component 
$J$ of $J(G)$, there exists a unique $\gamma \in \GN $ such that 
$J=J_{\gamma }(f).$}
\ \ \ \ \ \ \ \ \ \ \ \ \ \ \ \ \ \ \ \ \ \ \ \ \ \ \ \ \ \ \ 
\ \ \ \ \ \ \ \ \ \ \ \ \ \ 
\includegraphics[width=4.5cm,width=4.5cm]{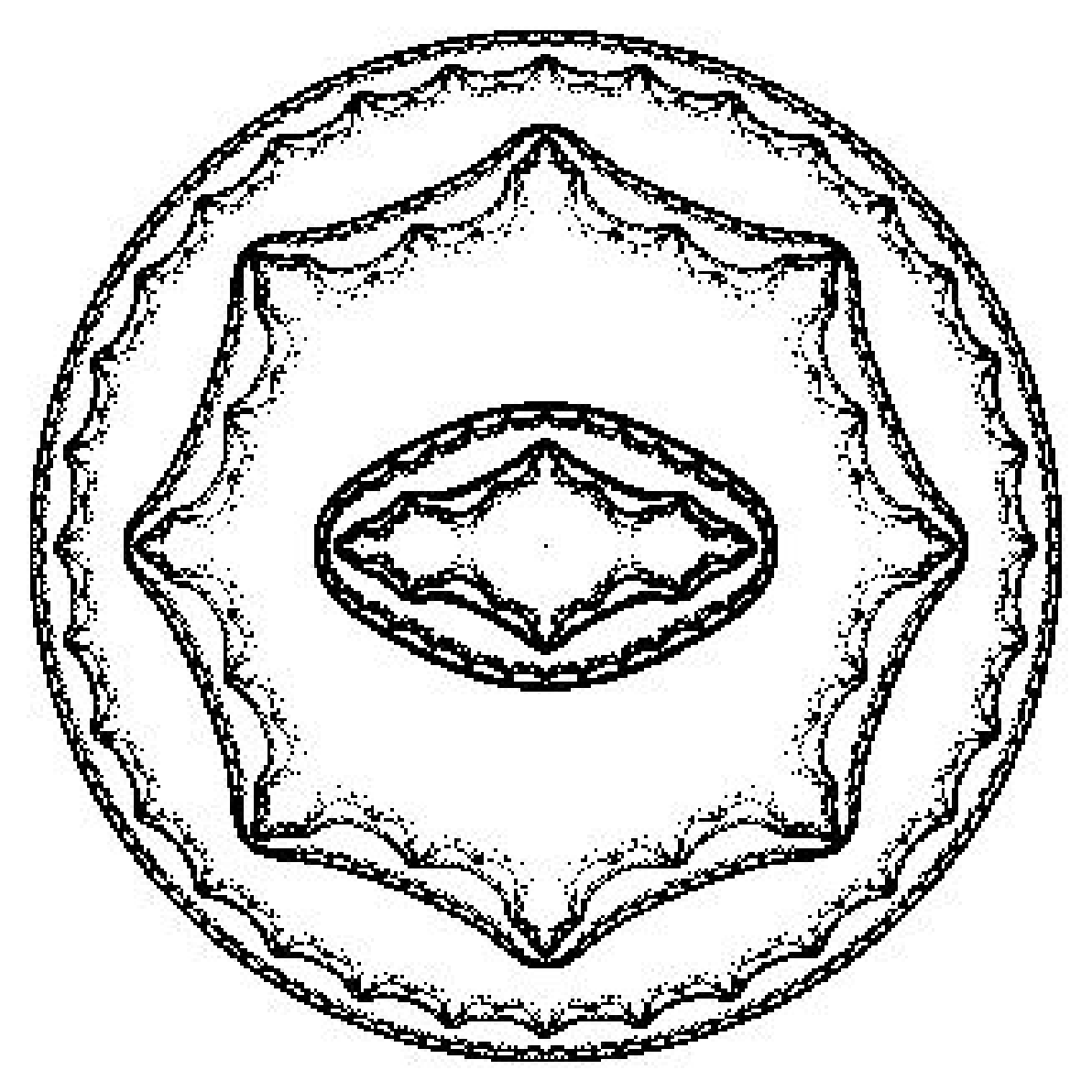}
\label{fig:dcgraph}
\end{figure}
\end{ex} 
\begin{ex}
\label{jbnqex}
Let $h_{1}(z):= z^{2}-1$ and $h_{2}(z):=az^{2}$, 
where $a\in \CC $ with $0<|a|<0.1.$ 
Let $\G := \{ h_{1},h_{2}\} .$ Moreover, 
let $G:=\langle h_{1},h_{2}\rangle .$ 
Let $U:= \{ |z|<0.2\} .$ Then, it is easy to see that 
$h_{2}(U)\subset U,\ h_{2}(h_{1}(U))\subset U,$ and 
$h_{1}^{2}(U)\subset U.$ Hence, $U\subset 
F(G).$ It follows that $P^{\ast }(G)\subset $ 
int$(\hat{K}(G))\subset F(G).$ Therefore, $G\in {\cal G}$ and 
$G$ is hyperbolic. Since $J(h_{1})$ is not a Jordan curve and 
$J(h_{2})$ is a Jordan curve, Theorem~\ref{mainthran2} implies that 
there exists a residual subset ${\cal U} $ of $ \GN $ such that,  
for each Borel probability measure $\tau $ 
on Poly$_{\deg \geq 2}$ with $\G _{\tau }=\G $, we have 
$\tilde{\tau }({\cal U})=1$, and such that, for each 
$\g \in {\cal U}$, 
$J_{\g }(f)$ is a Jordan curve but not a quasicircle. Moreover, 
for each $\g \in {\cal U}$, $A_{\g }(f)$ is a John domain, but 
the bounded component of $F_{\g }(f)$ is not a John domain. 
\end{ex} 
\begin{rem}
Let $h\in $ Poly$_{\deg \geq 2}$ be a polynomial. 
Suppose that $J(h)$ is a Jordan curve but not a quasicircle. 
Then, it is easy to see that there exists a parabolic fixed point 
of $h$ in $\CC $ and the bounded connected component of $F(h)$ is the 
immediate parabolic basin. Hence, $\langle h\rangle $ is not semi-hyperbolic.
Moreover, by \cite{CJY}, 
$F_{\infty }(h)$ is not a John domain. 

 Thus what we see in statement \ref{mainthran1-2} in Theorem~\ref{mainthran1} and statement \ref{mainthran2-2} in Theorem~\ref{mainthran2}, 
 as illustrated in 
Example ~\ref{jbnqexfirst} and 
Example~\ref{jbnqex} (see also Section \ref{Const}), is a 
 phenomenon which can hold in the {\em random} dynamics 
 of a family of polynomials, but cannot hold in the usual 
 iteration dynamics of a single polynomial. 
 Namely, it can hold that for almost every $\g \in \GN $, 
$J_{\g }(f)$ is a Jordan curve and fails to be a quasicircle 
all while the basin of infinity $A_{\g }(f)$ is still a John domain. 
Whereas, if $J(h)$, for some polynomial $h$, is a Jordan curve which 
fails to be a quasicircle, then the basin of infinity $F_{\infty }(h)$ 
is necessarily {\bf not} a John domain.   

 In Section~\ref{Const}, we will see how easily  
situation \ref{mainthran1-2} in Theorem~\ref{mainthran1} and situation \ref{mainthran2-2} in Theorem~\ref{mainthran2} 
occur.
 
 Pilgrim and Tan Lei (\cite{PT}) showed that there exists a  
hyperbolic rational map $h$ with {\em disconnected} Julia set such that  
``almost every'' connected component of $J(h)$ 
is a Jordan curve but not a quasicircle.  
\end{rem}
We give a sufficient condition so that statement~\ref{mainthran2-1} in Theorem~\ref{mainthran2} holds. 
\begin{prop}
\label{ranprop1}
Let $\G $ be a non-empty compact subset of {\em Poly}$_{\deg \geq 2}.$ 
Let $f:\GN \times \CCI \rightarrow \GN \times \CCI $ be the 
skew product associated with the family $\G .$ 
Let $G$ be the polynomial semigroup generated by $\G .$ 
Suppose that $P^{\ast }(G)$ is included in a connected component 
of {\em int}$(\hat{K}(G)).$ Then, there exists a constant $K\geq 1$ such that 
for each $\g \in \GN $, $J_{\g }(f)$ is a $K$-quasicircle.
\end{prop}
\begin{ex}
Let $d_{1},\ldots ,d_{m}\in \NN $ with $d_{j}\geq 2$ for each $j$, and 
let $h_{j}(z)=a_{j}z^{d_{j}}+c_{j}, a_{j}\neq 0,$ for each $j=1,\ldots, m.$ 
Let $\G =\langle g_{1},\ldots ,g_{m}\rangle .$ 
If $|c_{j}|$ is small enough for each $j$, then 
$\G $ satisfies the assumption of Proposition~\ref{ranprop1}. Thus 
statement~\ref{mainthran2-1} in Theorem~\ref{mainthran2} holds. 
\end{ex}
We have also many examples of $\Gamma $ such that 
 statement~\ref{mainthran1-3} in Theorem~\ref{mainthran1} or 
statement~\ref{mainthran2-3} in Theorem~\ref{mainthran2} holds. 
\begin{ex}
Let $h_{1}\in \mbox{Poly}_{\deg \geq 2}.$ 
Suppose that $\langle h_{1}\rangle \in {\cal G}$ and 
$h_{1}$ is hyperbolic. Suppose also that 
$h_{1}$ has at least two attracting periodic points in $\CC .$ 
Let $\Gamma $ be a small compact neighborhood of $h_{1}$ in 
Poly$_{\deg \geq 2}.$ 
Then $\langle \G \rangle \in {\cal G}$  and $\langle \G \rangle $ is hyperbolic 
(see Lemma~\ref{l:omegahopen}). Moreover,  
 by the argument in the proof of Lemma~\ref{l:disjfinv}, we see that 
for each $\g \in \GN $, $F_{\g }(f)$ has at least two bounded connected components, 
where $f:\GN \times \CCI \rightarrow \GN \times \CCI $ is the skew product associated with 
$\Gamma $. Thus statement~\ref{mainthran2-3} in Theorem~\ref{mainthran2} holds.    
We remark that by using Lemma~\ref{l:twohypp}, \ref{l:disjfinv} and their proofs, 
 we easily obtain many examples of $\Gamma $ such that 
 statement~\ref{mainthran1-3} in Theorem~\ref{mainthran1} or 
statement~\ref{mainthran2-3} in Theorem~\ref{mainthran2} holds. 
\end{ex}
\section{Tools}
\label{Tools}
To show the main results, we need 
some tools in this section.
\subsection{Fundamental properties of rational semigroups}
{\bf Notation:} 
For a rational semigroup $G$, we set 
$E(G):=\{ z\in \CCI \mid \sharp (\bigcup _{g\in G}g^{-1}(\{ z\} ))<\infty \} .$ 
This is called the exceptional set of $G.$\\ 
{\bf Notation:} Let $r>0.$ 
For a subset $A$ of $\CCI $, we set 
$B(A,r):= \{ z\in \CCI \mid d_{s}(z,A)<r\} $, where $d_{s}$ is the spherical distance. 
For a subset $A$ of $\CC $, we set 
$D(A,r):= \{ z\in \CC \mid d_{e}(z,A)<r \}$, where $d_{e}$ is the Euclidean distance. 

We use the following Lemma~\ref{hmslem} and Theorem~\ref{repdense} in the proofs of the main results.  
\begin{lem}[\cite{HM1,GR,S3,S1}]
\label{hmslem}
Let $G$ be a rational semigroup.\
\begin{enumerate}
\item 
\label{invariant}
For each $h\in G,\ $ we have 
$h(F(G))\subset F(G)$ and $h^{-1}(J(G))\subset J(G).$ Note that we do not 
have that the equality holds in general.
\item
\label{bss}
If $G=\langle h_{1},\ldots ,h_{m}\rangle $, then 
$J(G)=h_{1}^{-1}(J(G))\cup \cdots \cup h_{m}^{-1}(J(G)).$ 
More generally, if $G$ is generated by a compact subset 
$\G $ of {\em Rat}, then 
$J(G)=\bigcup _{h\in \G}h^{-1}(J(G)).$ 
(We call this property of the Julia set of a compactly generated rational semigroup ``backward self-similarity." )
\item
\label{Jperfect}
If \( \sharp (J(G)) \geq 3\) ,\ then \( J(G) \) is a 
perfect set.\ 
\item
\label{egset}
If $\sharp (J(G))\geq 3$ ,\ then 
$ \sharp (E(G)) \leq 2. $
\item
\label{o-set}
If a point \( z\) is not in \( E(G),\ \) then 
 \( J(G)\subset \overline{\bigcup _{g\in G}g^{-1}(\{ z\} )} .\) In particular 
if a point\ \( z\)  belongs to \ \( J(G)\setminus E(G), \) \ then
$ \overline{\bigcup _{g\in G}g^{-1}(\{ z\})}=J(G). $
\item
\label{backmin}
If \( \sharp (J(G))\geq 3 \) ,\ 
then
\( J(G) \) is the smallest closed backward invariant set containing at least
three points. Here we say that a set $A$ is backward invariant under $G$ if
for each \( g\in G,\ g^{-1}(A)\subset A.\ \)
\end{enumerate}
\end{lem}
\begin{thm}[\cite{HM1,GR,S3}]
\label{repdense}
Let $G$ be a rational semigroup.
If $\sharp (J(G))\geq 3$, then \\ 
$ J(G)=\overline{\{ z\in \CC \mid 
\exists g \in G,\ g(z)=z,\ |g'(z)|>1 \} } $, where 
the closure is taken in $\CCI .$  
In particular,\ $J(G)=\overline{\bigcup _{g\in G}J(g)}.$
\end{thm}
\begin{rem}
If a rational semigroup $G$ contains an element $g$ with $\deg (g)\geq 2$, 
then $\sharp (J(g))\geq 3$, which implies that $\sharp (J(G))\geq 3.$  
\end{rem}
\subsection{Fundamental properties of fibered rational maps}
\begin{lem}
\label{fibfundlem}
Let $f:X\times \CCI \rightarrow 
X\times \CCI $ be a rational skew product 
over $g:X\rightarrow X.$ Then, 
we have the following.
\begin{enumerate}
\item \label{fibfundlem1}
{\bf (\cite[Lemma 2.4]{S1})}
For each $x\in X$, 
$(f_{x,1})^{-1}(J_{g (x)}(f))
=J_{x}(f).$ Furthermore, we have 
$\hat{J}_{x}(f)\supset 
J_{x}(f).$ Note that 
{\bf equality $\hat{J_{x}}(f)=J_{x}(f)$ does not 
hold in general.}

 If $g:X\rightarrow X$ is 
a surjective and open map, then
$f^{-1}(\tilde{J}(f))=\tilde{J}(f)=f(\tilde{J}(f))$, and 
 for each $x\in X$, 
$(f_{x,1})^{-1}(\hat{J}_{g (x)}(f))=
\hat{J}_{x}(f).$
\item \label{fibfundlem2}
({\bf \cite{J, S1}})
If $d(x)\geq 2$ 
for each $x\in X $,  
then 
for each $x\in X$, $J_{x}(f)$ is a non-empty perfect set with $\sharp (J_{x}(f))\geq 3.$
Furthermore, the map $x\mapsto J_{x}(f)$ is 
lower semicontinuous; i.e.,  
for any point $(x,y)\in X\times \CCI $ 
with $y\in J_{x}(f)$ and 
any sequence $\{ x^{n}\} _{n\in \NN }$ in $X$ 
with $x^{n}\rightarrow x,\ $  
there exists a sequence $\{ y^{n}\} _{n\in \NN }$ 
in $\CCI $ with 
$y^{n}\in J_{x^{n}}(f)$ for each $n\in \NN $ such that 
$y^{n}\rightarrow y.$ However,  
$x\mapsto J_{x}(f)$ is {\bf not}
 continuous with respect to the Hausdorff topology in 
general.
\item \label{fibfundlemast}
If $d(x)\geq 2$ for each $x\in X$, 
then $\inf _{x\in X}$diam$_{S}J_{x}(f)>0$, 
where diam$_{S}$ denotes the diameter with respect to the 
spherical distance.
\item \label{fibfundlem4}
If $f:X\times \CCI \rightarrow 
X\times \CCI $ is a polynomial skew product 
and $d(x)\geq 2$     
for each $x\in X $,  
then  there exists a ball $B$ 
around $\infty $ such that for each $x\in 
X$, $B\subset A_{x}(f)\subset F_{x}(f)$, and  
for each $x\in X$, 
$J_{x}(f)=\partial K_{x}(f)=\partial A_{x}(f) .$
Moreover, for each $x\in X$, 
$A_{x}(f)$ is connected.  
\item \label{fibfundlema}
If $f:X\times \CCI \rightarrow 
X\times \CCI $ is a polynomial skew product 
and $d(x)\geq 2$     
for each $x\in X $, and if 
$\omega \in X$ is a point such that 
{\em int}$(K_{w}(f))$ is a non-empty 
set, then 
$\overline{\mbox{{\em int}}(K_{\omega }(f))}=K_{\omega }(f)$ 
and $\partial (\mbox{{\em int}}(K_{\omega }(f)))=J_{\omega }(f).$ 
\end{enumerate}
\end{lem}
\begin{proof}
For the proof of statement \ref{fibfundlem1}, 
see \cite[Lemma 2.4]{S1}.
For the proof of statement \ref{fibfundlem2}, 
see \cite{J} and \cite{S1}. 

 By statement \ref{fibfundlem2}, 
it is easy to see that statement \ref{fibfundlemast} holds. 
Moreover, 
it is easy to see that 
statement \ref{fibfundlem4} holds. 

 To show statement \ref{fibfundlema}, 
let $y\in J_{\omega }(f)$ be a point. 
Let $V$ be an arbitrary 
neighborhood of $y$ in $\CCI .$ 
Then, by the self-similarity of Julia sets (see \cite{Bu1}), 
there exists an $n\in \NN $ such that 
$f_{\omega ,n}(V\cap J_{\omega }(f))=J_{g^{n}(\omega )}(f).$ 
Since $\partial (\mbox{int}(K_{g^{n}(\omega )}(f)))\subset 
J_{g^{n}(\omega )}(f)$ and $(f_{\omega ,n})^{-1}
(K_{g^{n}(\omega )}(f))=K_{\omega }(f)$, 
it follows that 
$V\cap \partial (\mbox{int}(K_{\omega }(f)))\neq \emptyset .$ 
Hence, we obtain $J_{\omega }(f)=\partial 
(\mbox{int}(K_{\omega }(f)).$ Therefore, we have proved 
statement \ref{fibfundlema}.  
\end{proof}
\begin{lem}
\label{fiblem}
Let $f:\GN \times 
\CCI \rightarrow \GN \times 
\CCI $ be a skew product  
associated with a compact subset 
$\G $ of {\em Rat}. 
Let $G$ be the rational semigroup generated by $\G .$ 
Suppose that $\sharp (J(G))\geq 3.$ 
Then, we have the following. 
\begin{enumerate}
\item  \label{pic}
$\pi _{\CCI }(\tilde{J}(f))=J(G).$
\item \label{fibfundlem5}
For each $\g =(\g _{1},\g _{2},\ldots ,)\in \GN $, 
$\hat{J}_{\g }(f)=\bigcap _{j=1}^{\infty }\g_{1}^{-1}
\cdots \g_{j}^{-1}(J(G)).$
\end{enumerate}
\end{lem}
\begin{proof}
First, we show statement \ref{pic}. 
Since $J_{\g }(f)\subset J(G)$ for each $\g \in \G $, 
we have $\pi _{\CCI }(\tilde{J}(f))\subset J(G).$ 
By Theorem~\ref{repdense}, we have 
$J(G)=\overline{\bigcup _{g\in G}J(g)}.$ 
Since $\bigcup _{g\in G}J(g)\subset \pi _{\CCI }(\tilde{J}(f))$, 
we obtain $J(G)\subset \pi _{\CCI }(\tilde{J}(f)).$ 
Therefore, we obtain $\pi _{\CCI }(\tilde{J}(f))=J(G).$ 

 We now show statement \ref{fibfundlem5}. 
Let $\g =(\g _{1},\g _{2},\ldots )\in \GN .$ 
By statement \ref{fibfundlem1} in Lemma~\ref{fibfundlem}, 
we see that for each $j\in \NN $, 
$\g_{j}\cdots \g_{1}(
\hat{J}_{\g }(f))=
\hat{J}_{\sigma ^{j}(\g )}(f)
\subset J(G).$ 
Hence, 
$\hat{J}_{\g }(f)\subset 
\bigcap _{j=1}^{\infty }
\g _{1}^{-1}\cdots 
\g _{j}^{-1}(J(G)).$
 Suppose that there exists a point 
 $(\g ,y)\in \GN \times \CCI $ such that  
 $y\in $ $\left( \bigcap _{j=1}^{\infty }
\g _{1}^{-1}\cdots 
\g _{j}^{-1}(J(G))\right) \setminus 
\hat{J}_{\g }(f).$ 
Then, we have 
$(\g ,y)\in (\GN \times \CCI ) 
\setminus \tilde{J}(f).$ 
Hence, there exists a 
neighborhood $U$
of $\g $ in $\GN $ and a 
neighborhood $V$ of $y$ in $\CCI $ such that 
$U\times V\subset \tilde{F}(f).$ 
Then, there exists an $n\in \NN $ such that 
$\{ \rho \in \GN \mid \rho _{j}=\gamma _{j}, j=1,\ldots ,n\} \subset U.$  
Combining it with Lemma~\ref{fibfundlem}-\ref{fibfundlem1}, 
we obtain $\tilde{F}(f)\supset 
f^{n}(U\times V)\supset \GN \times 
\{ f_{\g ,n}(y)\} .$ 
Moreover, 
since we have 
$f_{\g ,n}(y)\in J(G)=
\pi _{\CCI }(\tilde{J}(f))$, 
where the last equality holds 
by statement \ref{pic}, 
we get that there exists an element 
$\g '\in \GN $ such that 
$(\g ', f_{\g ,n}(y))\in 
\tilde{J}(f).$ However, it contradicts 
$(\g ',f_{\g ,n}(y))\in 
\GN \times \{ f_{\g ,n}(y)\} 
\subset \tilde{F}(f).$ 
Hence, we obtain 
$\hat{J}_{\g }(f)= 
\bigcap _{j=1}^{\infty }
\g_{1}^{-1}\cdots 
\g_{j}^{-1}(J(G)).$
\end{proof}

\begin{lem}
\label{fibconnlem}
Let $f:X\times \CCI \rightarrow 
X\times \CCI $ be a polynomial skew product 
over $g:X\rightarrow X$ such that  
for each $x\in X$, 
$d(x)\geq 2.$ Then, 
the following are equivalent.
\begin{enumerate}
\item \label{fcc1}
$\pi _{\CCI }(P(f))\setminus \{ \infty \} $ 
is bounded in $\CC .$ 
\item \label{fcc2}
For each $x\in X$, $J_{x}(f)$ is 
connected.
\item \label{fcc3}
For each $x\in X$, $\hat{J}_{x}(f)$ is 
connected.
\end{enumerate} 
\end{lem}
\begin{proof} 
First, we show \ref{fcc1} $\Rightarrow $\ref{fcc2}. 
Suppose that \ref{fcc1} holds.
Let 
$R>0$ be a number such that 
for each $x\in X$, 
$B:=\{ y\in \CCI 
\mid |y|>R\} \subset A_{x}(f)$ 
and $\overline{f_{x,1}(B)}\subset B.$ 
Then, for each $x\in X$, we have 
$A_{x}(f)=\bigcup _{n\in \NN }(f_{x,n})^{-1}(B)$ 
and $(f_{x,n})^{-1}(B)\subset 
(f_{x,n+1})^{-1}(B)$, 
for each $n\in \NN .$ Furthermore, since 
we assume \ref{fcc1}, we see that 
for each $n\in \NN $, $(f_{x,n})^{-1}(B)$ 
is a simply connected domain, by the Riemann-Hurwitz formula. 
Hence, for each $x\in X$, 
$A_{x}(f)$ is a simply connected domain.   
Since $\partial A_{x}(f)=J_{x}(f)$ for each $x\in X,$ 
we conclude that for each $x\in X$, 
$J_{x}(f)$ is connected. 
Hence, we have shown \ref{fcc1} $\Rightarrow $ \ref{fcc2}. 

 Next, we show 
 \ref{fcc2} $\Rightarrow $ \ref{fcc3}. 
 Suppose that \ref{fcc2} holds. 
 Let $z_{1}\in \hat{J}_{x}(f)$ and $z_{2}\in J_{x}(f)$ be two points. 
Let 
 $\{ x^{n}\} _{n\in \NN }$ be a sequence 
 such that $x^{n}\rightarrow x$ as $n\rightarrow \infty $, and 
 such that 
 $d(z_{1},J_{x^{n}}(f))\rightarrow 0$ as 
 $n\rightarrow \infty .$ We may assume that 
 there exists a non-empty compact set $K$ in 
 $\CCI $ such that 
 $J_{x^{n}}(f)\rightarrow K$ as $n\rightarrow \infty $, 
 with respect to the Hausdorff topology in the 
 space of non-empty compact sets in $\CCI .$ 
 Since we assume \ref{fcc2}, 
 $K$ is connected. 
 By Lemma~\ref{fibfundlem}-\ref{fibfundlem2}, 
 we have $d(z_{2}, J_{x^{n}}(f))\rightarrow 0$ as 
 $n\rightarrow \infty .$ Hence, 
 $z_{i}\in K$ for each $i=1,2.$ 
 Therefore, $z_{1}$ and $z_{2}$ belong to the same connected 
 component of $\hat{J}_{x}(f).$ Thus, 
 we have shown \ref{fcc2} $\Rightarrow $ \ref{fcc3}. 

 Next, we show \ref{fcc3} $\Rightarrow $ \ref{fcc1}. 
 Suppose that \ref{fcc3} holds.
 It is easy to see that 
 $A_{x}(f)\cap \hat{J}_{x}(f)=\emptyset $ for each 
 $x\in X.$ 
 Hence, $A_{x}(f)$ is a connected component of 
 $\CCI \setminus \hat{J}_{x}(f).$ 
 Since we assume \ref{fcc3}, we have that 
 for each $x\in X$, 
 $A_{x}(f)$ is a simply connected domain. 
 Since $(f_{x,1})^{-1}(A_{g(x)}(f))=
 A_{x}(f)$, the Riemann-Hurwitz formula implies that 
 for each $x\in X$, 
 there exists no critical point of 
 $f_{x,1}$ in $A_{x}(f)\cap \CC .$ 
 Therefore, we obtain \ref{fcc1}. 
 Thus, we have shown \ref{fcc3} $\Rightarrow $ \ref{fcc1}.
\end{proof}   
\begin{cor}
\label{fibconncor}
Let $G=\langle h_{1},h_{2}\rangle \in {\cal G}.$ Then, 
$h_{1}^{-1}(J(h_{2}))$ is connected.
\end{cor} 
\begin{proof}
Let $f:\GN \times \CCI \rightarrow \GN \times \CCI $ 
be the skew product associated with the family 
$\G =\{ h_{1},h_{2}\} .$ 
Let $\g =(h_{1},h_{2},h_{2},h_{2},h_{2},\ldots )\in \GN .$ 
Then, by Lemma~\ref{fibfundlem}-\ref{fibfundlem1},
 we have $J_{\g }(f)=h_{1}^{-1}(J(h_{2})).$ 
From Lemma~\ref{fibconnlem}, it follows that 
$h_{1}^{-1}(J(h_{2}))$ is connected.
\end{proof} 
 
\begin{lem}
\label{fibconnlem2}
Let $G$ be a polynomial semigroup
generated by a compact subset $\G $ of 
{\em Poly}$_{\deg \geq 2}.$ 
Let $f:\GN \times \CCI 
\rightarrow \GN \times \CCI $ be 
the skew  product associated with the 
family $\G .$
Suppose that  
$G\in {\cal G}.$ Then 
for each $\g =(\g _{1},\g _{2},\ldots ,)\in 
\GN $, 
the sets $J_{\g }(f),\ \hat{J} _{\g }(f)$, and 
$\bigcap _{j=1}^{\infty }
 \g _{1}^{-1}\cdots \g _{j}^{-1}(J(G))$ 
 are connected.
\end{lem}
\begin{proof}
From Lemma~\ref{fiblem}-\ref{fibfundlem5} and 
Lemma~\ref{fibconnlem}, the lemma follows.
\end{proof}
\begin{lem}
\label{fiborder}
Under the same assumption as that in Lemma~\ref{fibconnlem2}, 
let $\g ,\rho \in \GN $ be two elements with 
$J_{\g }(f)\cap J_{\rho }(f)=\emptyset .$ 
Then, either $J_{\g }(f)< J_{\rho }(f)$ 
or $J_{\rho }(f)< J_{\g }(f).$
\end{lem}
\begin{proof}
Let $\g ,\rho \in \GN $ with 
$J_{\g }(f)\cap J_{\rho }(f)=\emptyset .$ 
Suppose that the statement 
``either $J_{\g }(f)<J_{\rho }(f)$ or $J_{\rho }(f)<J_{\g }(f)$" 
is not true. 
Then, Lemma~\ref{fibconnlem} implies that 
$J_{\g }(f)$ is included in the unbounded 
component of $\CC \setminus J_{\rho }(f)$, and that 
$J_{\rho }(f)$ is included in the unbounded 
component of $\CC \setminus J_{\g}(f).$ 
From Lemma~\ref{fibfundlem}-\ref{fibfundlem4}, it follows that 
$K_{\rho }(f)$ is included in the unbounded 
component $A_{\g }(f)\setminus \{ \infty \} $ 
of $\CC \setminus J_{\g }(f).$ 
However, it causes a contradiction, since 
$\emptyset \neq P^{\ast }(G)\subset \hat{K}(G)\subset 
K_{\rho }(f)\cap K_{\g }(f).$  
\end{proof}
\begin{df}
Let 
$f:\GNCR $ be a polynomial skew product 
over $g:X\rightarrow X.$ 
Let $p\in \CC $ and 
$\epsilon >0.$ 
We set \\ 
${\cal F}_{f,p,\epsilon }:= 
\{\alpha :D(p,\epsilon )\rightarrow \CC \mid 
\alpha \mbox{ is a well-defined branch of }
(f_{x,n})^{-1}, x\in X, n\in \NN \} .$  
\end{df}
\begin{lem}
\label{invnormal}
Let $f:\GNCR $ be a polynomial skew product over 
$g:X\rightarrow X$ such that for each $x\in X$, 
$d(x)\geq 2.$ 
Let $R>0,\epsilon >0$, and \\ 
${\cal F}:=
\{ \alpha \circ \beta :D(0,1)\rightarrow \CC 
\mid 
\beta :D(0,1)\cong D(p,\epsilon ),\ 
\alpha :D(p,\epsilon )\rightarrow \CC ,\ 
\alpha \in {\cal F}_{f,p,\epsilon },\ p\in D(0,R)\} .$ 
Then, ${\cal F} $ is normal in $D(0,1).$ 
\end{lem}
\begin{proof}
Since $d(x)\geq 2$ for each $x\in X$, 
there exists a ball $B$ around $\infty $ 
with $B\subset \CCI \setminus D(0,R+\epsilon )$ such that 
for each $x\in X$, $f_{x,1}(B)\subset B.$ 
Let $p\in D(0,R).$ Then, for each 
$\alpha \in {\cal F}_{f,p,\epsilon }$, 
$\alpha (D(p,\epsilon ))\subset \CCI \setminus B.$ 
Hence, ${\cal F} $ is normal in $D(0,1).$ 
\end{proof}
\begin{df}
For a polynomial semigroup $G$ with 
$\infty \in F(G)$, we denote by 
$F_{\infty }(G)$ the connected component of $F(G)$ containing 
$\infty .$ Moreover, for a polynomial $g$ with 
$\deg (g)\geq 2$, we set $F_{\infty }(g):= 
F_{\infty }(\langle g\rangle ).$ (Note that if $\Gamma $ is a non-empty compact subset 
of  Poly$_{\deg \geq 2}$, then $\infty \in F(\langle \Gamma \rangle ).$)  
\end{df}
\begin{lem}
\label{constlimlem}
Let $G$ be a  
polynomial semigroup generated by a compact subset 
$\G $ of {\em Poly}$_{\deg \geq 2}.$ 
If a sequence $\{ g_{n}\} _{n\in \NN }$ of elements of $G$ tends to a 
constant $w_{0}\in \CCI $ locally uniformly on a domain $V\subset \CCI $, 
then $w_{0}\in P(G).$   
\end{lem}
\begin{proof}
Since $\infty \in P(G)$, we may assume that 
$w_{0}\in \CC .$ 
Suppose $w_{0}\in \CC \setminus P(G).$ 
Then, there exists a $\delta >0$ such that 
$B(w_{0}, 2\delta )\subset \CC \setminus P(G).$ 
Let $z_{0}\in V$ be a point. Then, for each large $n\in \NN $, 
there exists a well-defined branch 
$\alpha _{n}$
of $g_{n}^{-1}$ on $B(w_{0},2\delta )$ such that 
$\alpha _{n}(g_{n}(z_{0}))=z_{0}.$ 
Let $B:=B(w_{0},\delta ).$  
Since $\G $ is compact, there exists a connected component $F_{\infty }(G)$ of 
$F(G)$ containing $\infty .$ 
Let $C$ be a compact neighborhood of $\infty $ 
in $F_{\infty }(G).$   
Then, we must have that there exists a number $n_{0}$ such that  
$\alpha _{n}(B)\cap C=\emptyset $ for each $n\geq n_{0}$, 
since $g_{n}\rightarrow \infty 
$ uniformly on $C$ as $n\rightarrow \infty $, 
which follows from that 
$\deg (g_{n})\rightarrow \infty $ and 
local degree at $\infty $ of $g_{n}$ tends to 
$\infty $ as $n\rightarrow \infty .$ 
Hence, $\{ \alpha _{n} |_{B}\}_{n\geq n_{0}} $ is normal 
in $B.$ However, for a small 
$\epsilon $ so that $B(z_{0},2\epsilon )\subset V$, 
we have $g_{n}(B(z_{0},\epsilon ))\rightarrow w_{0}$ 
as $n\rightarrow \infty $, 
and this is a contradiction. Hence, we must have that 
$w_{0}\in P(G).$ 
\end{proof}
%
%
\section{Proofs of the main results}
\label{Proofs} 
In this section, we demonstrate the main results. 

We first need the following. 
\begin{thm} 
\label{hypskewqc}
({\bf Uniform fiberwise quasiconformal surgery})
Let $f:X\times \CCI \rightarrow X\times \CCI $ be a 
polynomial skew product over $g:X\rightarrow X$ such that 
for each $x\in X$, $d(x)\geq 2.$ 
Suppose that $f$ is hyperbolic and that 
$\pi _{\CCI }(P(f))\setminus \{ \infty \} $ is bounded in 
$\CC .$ Moreover, suppose that 
for each $x\in X$,\ 
{\em int}$(K_{x}(f))$ is connected. 
Then, there exists a constant $K$ such that 
for each $x\in X,\ J_{x}(f)$ is a $K$-quasicircle.   
\end{thm}
\begin{proof}
Step 1: 
By \cite[Theorem 2.14-(4)]{S1}, the map $x\mapsto J_{x}(f)$ is 
continuous with respect to the Hausdorff topology. 
Hence, there exists a positive constant 
$C_{1}$ such that,  
for each 
$x\in X,\ \inf \{ d(a,b)\mid a\in J^{x}(f),\ b\in \pi ^{-1}(\{ x\} )\cap P^{\ast }(f)\} >C_{1}$,  
where $P^{\ast }(f):= P(f)\setminus \pi _{\CCI }^{-1}(\{ \infty \} )$, 
and $d(\cdot ,\cdot )$ denotes the spherical distance, under the 
canonical identification $\pi ^{-1}(\{ x\} )\cong \CCI .$   
Moreover, from the assumption, we have that,  
for each $x\in X$, int$(K_{x}(f))\neq \emptyset .$ 
Since $X$ is compact, 
it follows that, for each $x\in X,$  
there exists an analytic 
Jordan curve 
$\zeta _{x}$ in 
$K^{x}(f)\cap F^{x}(f)$ 
such that: 
\begin{enumerate}
\item 
$\pi ^{-1}(\{ x\} )\cap P^{\ast }(f)$ is included in the 
bounded component $V_{x}$ of $\pi ^{-1}(\{ x\} )\setminus 
\zeta _{x}$;
\item $\inf _{z\in \zeta _{x}}
d(z,\ J^{x}(f)\cup (\pi ^{-1}(\{ x\} )\cap P^{\ast }(f)))
\geq C_{2}$, where $C_{2}$ is a positive constant 
independent of $x\in X$; and  
\item there exist finitely many Jordan curves 
$\xi _{1},\ldots ,\xi _{k}$ in $\CC $ such that 
for each $x\in X$, there exists a $j$ with 
$\pi _{\CCI }(\zeta _{x})=\xi _{j}.$ 
\end{enumerate}
Step 2: 
By 
\cite[Corollary 2.7]{S4}, 
there exists an $n\in \NN $ such that 
for each $x\in X,\ W_{x}:=(f_{x}^{n})^{-1}(V_{g ^{n}(x)})
\supset \overline{V_{x}}$, 
$\inf \{ d(a,b)\mid a\in \partial W_{x},b\in \partial V_{x}, x\in X\} >0$, 
 and 
mod $(W_{x}\setminus \overline{V_{x}})\geq C_{3},$ where 
$C_{3}$ is a positive constant independent of $x\in X.$ 
In order to prove the theorem, 
since $J_{x}(f^{n})=J_{x}(f)$ for each $x\in X$, 
replacing $f:X\times \CCI \rightarrow X\times \CCI $ by 
$f^{n}:X\times \CCI \rightarrow X\times \CCI $, 
we may assume $n=1.$ 

\noindent Step 3: For each $x\in X$, let 
$\varphi _{x}:\pi ^{-1}(\{ x\} )\setminus 
\overline{V_{x}}\rightarrow 
\pi ^{-1}(\{ x\} )\setminus 
\overline{D(0,\frac{1}{2})}$ be 
a biholomorphic map such that 
$\varphi _{x}(x,\infty )=(x,\infty ),$ 
under the canonical identification $\pi ^{-1}(\{ x\} )\cong \CCI .$ 
We see that $\varphi _{x}$ extends 
analytically over $\partial V_{x}=\zeta _{x}.$ 
For each $x\in X$, we define a quasi-regular map 
$h_{x}: \pi ^{-1}(\{ x\} )\cong \CCI \rightarrow 
\pi ^{-1}(\{ g (x)\} )\cong \CCI $ as 
follows:
$$h_{x}(z):=\begin{cases}
            \varphi _{g (x)}f_{x}\varphi _{x}^{-1}(z),\ 
           \mbox{ if } z\in \varphi _{x}(\pi ^{-1}(\{ x\} )\setminus W_{x}),\\ 
            z^{d(x)},\ \mbox{ if } z\in \overline{D(0,\frac{1}{2})},\\ 
            \tilde{h}_{x}(z),\ \mbox{\ if } z\in \varphi _{x}(W_{x}\setminus 
            \overline{V_{x}}),
            \end{cases}
$$
where 
$\tilde{h}_{x}:\varphi _{x}(W_{x}\setminus \overline{V_{x}}) 
\rightarrow 
D(0,\frac{1}{2})\setminus 
\overline{D(0,(\frac{1}{2})^{d(x)})}$ is a 
regular covering and 
a $K_{0}$-quasiregular map with dilatation constant 
$K_{0}$ independent of $x\in X.$ 

\noindent Step 4: 
For each $x\in X$, we define a 
Beltrami differential 
$\mu _{x}(z)\frac{d\overline{z}}{dz}$ on 
$\pi ^{-1}(\{ x\} )\cong \CCI $ as follows: 
$$\begin{cases}
   \frac{\partial _{\overline{z}}\tilde{h}_{x}}
        {\partial _{z}\tilde{h}_{x}}
  \frac{d\overline{z}}{dz},\   
  \mbox{ if }z\in \varphi _{x}(W_{x}\setminus 
  \overline{V_{x}}),\\
  (h_{g ^{m-1}(x)}\cdots h_{x})^{\ast }
  (\frac{\partial _{\overline{z}}\tilde{h}_{g ^{m}(x)}}
        {\partial _{z}\tilde{h}_{g ^{m}(x)}}
  \frac{d\overline{z}}{dz}),\   
  \mbox{ if } z\in (h_{g ^{m-1}(x)}\cdots h_{x})^{-1}
   (\varphi _{g^{m}(x)}(W_{g ^{m}(x)}\setminus 
   \overline{V_{g^{m}(x)}})),\\ 
   
  0,\ \ \ \mbox{otherwise.}     
  \end{cases}
$$
Then, there exists a constant $k$ with $0<k<1$ such that 
for each $x\in X,\ \| \mu _{x}\| _{\infty }\leq k.$ 
By the construction, we have 
$h_{x}^{\ast }(\mu _{g (x)}\frac{d\overline{z}}{dz})
=\mu _{x}\frac{d\overline{z}}{dz}$, for each $x\in X.$ 
By the measurable Riemann mapping theorem (\cite[page 194]{LV}), 
for each $x\in X$, 
there exists a quasiconformal map $\psi _{x}:\pi ^{-1}(\{ x\} )\rightarrow 
\pi ^{-1}(\{ x\} )$  
such that 
$\partial _{\overline{z}}\psi _{x}=\mu _{x}\partial _{z}\psi _{x},\ 
\psi _{x}(0)=0,\ \psi _{x}(1)=1, $ and $\psi _{x}(\infty )=\infty $, 
under the canonical identification $\pi ^{-1}(\{ x\} )\cong \CCI .$ 
For each $x\in X$, let 
$\hat{h}_{x}:=\psi _{g (x)}h_{x}\psi _{x}^{-1}
:\pi ^{-1}(\{ x\} )\rightarrow \pi ^{-1}(\{ g (x)\} ).$ 
Then, $\hat{h}_{x}$ is holomorphic on 
$\pi ^{-1}(\{ x\} ).$  By the construction, 
we see that 
$\hat{h}_{x}(z)=c(x)z^{d(x)}$, 
where $c(x)=\psi _{g (x)}h_{x}\psi _{x}^{-1}(1)
=\psi _{g(x)}h_{x}(1).$ 
Moreover, by the construction again, we see 
that there exists a positive constant $C_{4}$ such 
that for each $x\in X,\ \frac{1}{C_{4}}\leq |h_{x}(1)|\leq C_{4}.$ 
Furthermore, 
\cite[Theorem 5.1 in page 73]{LV} implies that 
under the canonical identification 
$\pi ^{-1}(\{ x\} )\cong \CCI $, 
the family $\{ \psi _{x}^{-1}\} _{x\in X}$ is normal in 
$\CCI .$ Therefore, it follows that 
there exists a positive constant $C_{5}$ such that 
for each $x\in X,\ \frac{1}{C_{5}}\leq |c(x)|\leq 
C_{5}.$ 
Let $\tilde{J}_{x}$ be the set of 
non-normality of the sequence 
$\{ \hat{h} _{g ^{m}(x)}\cdots \hat{h}_{x}\} _{m\in \NN }$ 
in $\pi ^{-1}(\{ x\} )\cong \CCI .$ 
Since $\hat{h}_{x}(z)=c(x)z^{d(x)}$ and 
$\frac{1}{C_{5}}\leq |c(x)|\leq C_{5}$ for each 
$x\in X$, we get that for each $x\in X$, 
$\tilde{J}_{x}$ is a round circle. 
Moreover, \cite[Theorem 5.1 in page 73]{LV} implies that 
 $\{ \psi _{x}\} _{x\in X}$ 
and $\{ \psi _{x}^{-1}\} _{x\in X}$ are normal in $\CCI $ 
(under the canonical identification $\pi ^{-1}(\{ x\} )\cong \CCI $). 
Combining it with \cite[Corollary 2.7]{S4},  
we see that for each $x\in X$,   
$J^{x}(f)=\varphi _{x}^{-1}(\psi _{x}^{-1}(\tilde{J}_{x}))$, 
and it follows that 
there exists a constant $K$ such that 
for each $x\in X,\ J_{x}(f)$ is a $K$-quasicircle. 

 Thus, we have proved Theorem~\ref{hypskewqc}.
\end{proof}
\begin{rem}
Theorem~\ref{hypskewqc} generalizes a result 
in \cite[TH\'{E}OR\`{E}ME 5.2]{Se}, 
where O. Sester investigated hyperbolic polynomial skew products 
$f:X\times \CCI \rightarrow X\times \CCI $ 
such that for each $x\in X$, $d(x)=2.$   
\end{rem}
We next need the notion of (fiberwise) external rays.
\begin{df}
Let $h$ be a polynomial with $\deg (h)\geq 2.$ Suppose that 
$J(h)$ is connected. Let $\psi $ be a biholomorphic map 
$\CCI \setminus \overline{D(0,1)}\rightarrow  
F_{\infty }(h)$ with $\psi (\infty )=\infty $ such that 
$\psi ^{-1}\circ h\circ \psi (z)=z^{\deg (h)}$, 
for each $z\in \CCI \setminus \overline{D(0,1)}.$ 
(For the existence of the biholomorphic map $\psi $, see 
\cite[Theorem 9.5]{M}.) For each 
$\theta \in \partial D(0,1)$, we set 
$T(\theta ):= \psi (\{ r\theta \mid 1<r\leq \infty \} ).$ 
This is called the external ray (for $K(h)$) 
with angle $\theta .$ 
\end{df}

\begin{lem}
\label{060416lema}
Let $f:X\times \CCI \rightarrow X\times \CCI $ be a polynomial skew product over 
$g:X\rightarrow X$ such that 
for each $x\in X$, 
$d(x)\geq 2.$ Let $\g \in X$ be a point. 
Suppose that $J_{\g }(f)$ is a Jordan curve. 
Then, for each $n\in \NN $, 
$J_{g^{n}(\g )}(f)$ is a Jordan curve. 
Moreover, for each $n\in \NN $, 
there exists no critical value of 
$f_{\g ,n}$ in $J_{g^{n}(\g )}(f).$ 
\end{lem}
\begin{proof}
Since 
$(f_{\g ,1})^{-1}(K_{g(\g )}(f))=K_{\g }(f)$, 
it follows that int$(K_{g(\g )}(f))$ is a non-empty 
connected set. 
Moreover, 
$J_{g(\g )}(f)=f_{\g ,1}(J_{\g }(f))$ is locally connected. 
Furthermore, by Lemma~\ref{fibfundlem}-\ref{fibfundlem4} and 
Lemma~\ref{fibfundlem}-\ref{fibfundlema}, 
$\partial (\mbox{int}(K_{g(\g )}(f)))=\partial (A_{g(\g )}(f))= J_{g(\g )}(f).$ 
Combining the above arguments and \cite[Lemma 5.1]{PT}, 
we get that $J_{g(\g )}(f)$ is a Jordan curve. 
Inductively, 
we conclude that for each $n\in \NN $, 
$J_{g^{n}(\g )}(f)$ is a Jordan curve. 

 Furthermore, applying the Riemann-Hurwitz formula 
 to the map 
 $f_{\g, n}: \mbox{int}(K_{\g }(f))\rightarrow 
 \mbox{int}(K_{g^{n}(\g )}(f))$, we obtain 
 $1+p=\deg (f_{\g ,n})$, where $p$ 
 denotes the cardinality of the critical 
 points of 
 $f_{\g, n}: \mbox{int}(K_{\g }(f))\rightarrow 
 \mbox{int}(K_{g^{n}(\g )}(f))$ counting multiplicities. 
Hence, $p=\deg (f_{\g ,n})-1.$ 
It implies that there exists no critical value of 
$f_{\g ,n}$ in 
$J_{g^{n}(\g )}(f).$     
\end{proof}
The following is the key lemma to prove the main results. 
\begin{lem}
\label{060416lemast}
Let $f:X\times \CCI \rightarrow X\times \CCI  $ be a polynomial skew product over 
$g:X\rightarrow X$ such that 
for each $x\in X$, $d(x)\geq 2.$ 
Let $\mu >0$ be a number. 
Then, there exists a number $\delta >0$ such that 
the following statement holds.
\begin{itemize}
\item 
Let $\omega \in X$ be any element and 
$p\in J_{\omega }(f)$ any point with 
$\min \{ |p-b|\mid (\omega, b)\in P(f), b\in \CC \} >\mu .$  
Suppose that $J_{\omega }(f)$ is connected. 
Let 
$\psi :\CCI \setminus \overline{D(0,1)} 
\rightarrow A_{\omega }(f)$ be a biholomorphic map 
with $\psi (\infty )=\infty .$ For each 
$\theta \in \partial D(0,1)$, 
let $T(\theta )=\psi (\{ r\theta \mid 
1<r\leq \infty \} ).$ 
Suppose that there exist two elements 
$\theta _{1},\theta _{2}\in \partial D(0,1)$ with $\theta _{1}\neq \theta _{2}$ such that, for each $i=1,2$,  
$T(\theta _{i})$ lands at $p.$ 
Moreover, suppose that 
a connected component $V$ of 
$\CCI \setminus 
(T(\theta _{1})\cup T(\theta _{2})\cup \{ p\} )$ 
satisfies that 
diam $(V\cap K_{\omega }(f))\leq \delta .$ 
Furthermore, let 
$\g \in X$ be any element and suppose that 
there exists a sequence $\{ n_{k}\} _{k\in \NN }$ 
of positive integers such that 
$g^{n_{k}}(\g )\rightarrow \omega $ as 
$k\rightarrow \infty .$ Then, 
$J_{\g }(f)$ is not a quasicircle. 
\end{itemize}
\end{lem}
\begin{proof}
Let $\mu >0.$ Let $R>0$ with 
$\pi _{\CCI }(\tilde{J}(f))\subset D(0,R).$ 
Combining Lemma~\ref{invnormal} and 
Lemma~\ref{fibfundlem}-\ref{fibfundlemast}, 
we see that there exists a $\delta _{0}>0$ with   
$0<\delta _{0}<\frac{1}{20}\min 
\{ \inf _{x\in X}\mbox{diam }J_{x}(f),\mu \} $ such that the 
following statement holds:
\begin{itemize}
\item 
Let $x\in X$ be any point and $n\in \NN $ any element. 
Let 
$p\in D(0,R)$ be any point with 
$\min \{ |p-b|\mid (g^{n}(x), b)\in P(f), b\in \CC \} >\mu .$ 
Let $\phi :D(p,\mu )\rightarrow \CC $ be any well-defined 
branch of 
$(f_{x,n})^{-1}$ on $D(p,\mu ).$ Let 
$A$ be any subset of $D(p,\frac{\mu }{2})$ with diam $A\leq \delta _{0}.$ 
Then, 
\begin{equation}
\label{060416lemastpfeq1}
\mbox{diam }\phi (A)\leq \frac{1}{10}\inf 
_{x\in X}\mbox{diam }J_{x}(f). 
\end{equation}
\end{itemize}
We set $\delta :=\frac{1}{10}\delta _{0}.$ 
Let $\omega \in X$ and $p\in J_{\omega }(f)$ with 
$\min \{ |p-b|\mid (\omega , b)\in P(f), b\in \CC \} >\mu .$ 
Suppose that $J_{\omega }(f)$ is connected and 
let $\psi :\CCI \setminus \overline{D(0,1)}\rightarrow 
A_{\omega }(f)$ be a biholomorphic map with 
$\psi (\infty )=\infty .$ 
Setting $T(\theta ):=\psi (\{ r\theta 
\mid 1<r\leq \infty \} )$ for each 
$\theta \in \partial D(0,1)$, 
suppose that there exist two elements 
$\theta _{1},\theta _{2}\in \partial D(0,1)$ with $\theta _{1}\neq \theta _{2}$  such that 
for each $i=1,2$, 
$T(\theta _{i})$ lands at $p.$ Moreover, 
suppose that a connected component 
$V$ of 
$\CCI \setminus (T(\theta _{1})\cup T(\theta _{2})\cup 
\{ p\} )$ satisfies that 
\begin{equation}
\label{060416lemastpfeq2}
\mbox{diam}(V\cap K_{\omega }(f))\leq \delta .
\end{equation}
Furthermore, let $\g \in X$ and suppose that 
there exists a sequence $\{ n_{k}\} _{k\in \NN }$ of 
positive integers such that 
$g^{n_{k}}(\g )\rightarrow \omega $ as 
$k\rightarrow \infty .$ 
 We now suppose that $J_{\g }(f)$ is a quasicircle, 
and we will deduce a contradiction. 
Since $g^{n_{k}}(\g )\rightarrow \omega $ as 
$k\rightarrow \infty $, 
we obtain 
\begin{equation}
\label{060416lemastpfeq3}
\max \{ d_{e}(b,K_{\omega }(f))\mid b\in J_{g^{n_{k}}(\g )}(f)\} 
\rightarrow 0 \mbox{ as } k\rightarrow \infty .
\end{equation}  
We take a point $a\in V\cap J_{\omega }(f)$ and fix it. 
By Lemma~\ref{fibfundlem}-\ref{fibfundlem2}, 
there exists a number $k_{0}\in \NN $ such that 
for each $k\geq k_{0}$, 
there exists a point $y_{k}$ satisfying that 
\begin{equation}
\label{060416lemastpfeq4}
y_{k}\in J_{g^{n_{k}}(\g )}(f)\cap D(a,\frac{|a-p|}{10k}).
\end{equation}
Let $V'$ be the connected component 
of $\CCI \setminus (T(\theta _{1})\cup T(\theta _{2})\cup 
\{ p\} )$ with $V'\neq V.$ 
Then, by \cite[Lemma 17.5]{M}, 
\begin{equation}
\label{060416lemastpfeq5}
V'\cap J_{\omega }(f)\neq \emptyset .
\end{equation}
Combining (\ref{060416lemastpfeq5}) and 
Lemma~\ref{fibfundlem}-\ref{fibfundlem2}, 
we see that there exists a $k_{1}(\geq k_{0})\in \NN $
 such that for each $k\geq k_{1}$, 
\begin{equation}
\label{060416lemastpfeq6}
V'\cap J_{g^{n_{k}}(\g )}(f)\neq \emptyset .
\end{equation}  
By assumption and Lemma~\ref{060416lema}, 
for each $k\geq k_{1}$, 
$J_{g^{n_{k}}(\g )}(f)$ is a Jordan curve. 
Combining it with (\ref{060416lemastpfeq4}) and 
(\ref{060416lemastpfeq6}), there exists a 
$k_{2}(\geq k_{1})\in \NN $ satisfying that 
for each $k\geq k_{2}$,  
there exists a smallest closed subarc 
$\xi _{k}$ of 
$J_{g^{n_{k}}(\g )}(f)\cong S^{1}$ such that 
$y_{k}\in \xi _{k}$, $\xi _{k}\subset \overline{V}, $ 
$\sharp (\xi _{k}\cap (T(\theta _{1})\cup T(\theta _{2})\cup 
\{ p\} ))=2,$ and such that 
$\xi _{k}\neq J_{g^{n_{k}}(\g )}(f).$ 
For each $k\geq k_{2}$, let 
$y_{k,1}$ and $y_{k,2}$ be the two points such that 
$\{ y_{k,1},y_{k,2}\} =
\xi _{k}\cap 
(T(\theta _{1})\cup T(\theta _{2})\cup \{ p\} ).$ 
Then, (\ref{060416lemastpfeq3}) implies that 
\begin{equation}
\label{060416lemastpfeq7}
y_{k,i}\rightarrow p \mbox{ as }k\rightarrow \infty , \mbox{ for each }i=1,2.
\end{equation}
Combining that 
$\xi _{k}\subset V\cup \{ y_{k,1},y_{k,2}\} $, 
(\ref{060416lemastpfeq3}), and 
(\ref{060416lemastpfeq2}), we get that 
there exists a $k_{3}(\geq k_{2})\in \NN $ such that 
for each $k\geq k_{3}$, 
\begin{equation}
\label{060416lemastpfeq8}
\mbox{diam }\xi _{k}\leq \frac{\delta _{0}}{2}.
\end{equation}
Moreover, combining (\ref{060416lemastpfeq4}) and 
(\ref{060416lemastpfeq7}), 
we see that there exists a constant 
$C>0$ such that for each $k\in \NN $ with $k\geq k_{3}$, 
\begin{equation}
\label{060416lemastpfeq9}
\mbox{diam }\xi _{k}>C.
\end{equation}
Combining (\ref{060416lemastpfeq7}), (\ref{060416lemastpfeq8}), 
and (\ref{060416lemastpfeq9}), 
we may assume that there exists a constant $C>0$ such that 
for each $k\in \NN $, 
\begin{equation}
\label{060416lemastpfeq10}
C<\mbox{ diam }\xi _{k}\leq \frac{\delta _{0}}{2}
\mbox{ and } 
\xi _{k}\subset D(p,\delta _{0}).
\end{equation}
By Lemma~\ref{060416lema}, each connected component $v$ of 
$(f_{\g ,n_{k}})^{-1}(\xi _{k})$ is a subarc of 
$J_{\g }(f)\cong S^{1}$ and 
$f_{\g ,n_{k}}:v\rightarrow \xi _{k}$ is a homeomorphism.
For each $k\in \NN $, 
let $\lambda _{k}$ be a connected component of 
$(f_{\g ,n_{k}})^{-1}(\xi _{k})$, and let 
$z_{k,1},z_{k,2}\in \lambda _{k}$ be the two endpoints 
of $\lambda _{k}$ such that 
$f_{\g ,n_{k}}(z_{k,1})=y_{k,1}$ and 
$f_{\g ,n_{k}}(z_{k,2})=y_{k,2}.$ Then, combining  
(\ref{060416lemastpfeq1}) and (\ref{060416lemastpfeq10}), 
we obtain 
\begin{equation}
\label{060416lemastpfeq11}
\mbox{ diam} \lambda _{k}<\mbox{ diam }(J_{\g }(f)\setminus 
\lambda _{k}), \mbox{ for each large }k\in \NN .
\end{equation}
Moreover, combining (\ref{060416lemastpfeq7}), 
(\ref{060416lemastpfeq10}), and Koebe distortion theorem, 
it follows that 
\begin{equation}
\label{060416lemastpfeq12}
\frac{\mbox{diam }\lambda _{k}}{|z_{k,1}-z_{k,2}|}\rightarrow \infty 
\mbox{ as }k\rightarrow \infty .
\end{equation}
Combining (\ref{060416lemastpfeq11}) and (\ref{060416lemastpfeq12}), 
we conclude that $J_{\g }(f)$ cannot be a quasicircle, 
since we have the following well-known fact:\\ 
Fact (\cite[Chapter 2]{LV}): 
Let $\xi $ be a Jordan curve in $\CC .$ Then, 
$\xi $ is a quasicircle if and only if 
there exists a constant $K>0 $ such that 
for each $z_{1},z_{2}\in \xi $ with $z_{1}\neq z_{2}$, 
we have $\frac{\mbox{diam }\lambda (z_{1},z_{2})}{|z_{1}-z_{2}|}\leq K$, 
where $\lambda (z_{1},z_{2})$ denotes the smallest closed subarc 
of $\xi $ such that 
$z_{1},z_{2}\in \lambda (z_{1},z_{2})$ and such that 
diam $\lambda (z_{1},z_{2})<$ 
diam $(\xi \setminus \lambda (z_{1},z_{2})).$ 

 Hence, we have proved Lemma~\ref{060416lemast}. 
\end{proof}
We now give some sufficient conditions for a fiberwise Julia set to be a Jordan curve. 
\begin{prop}
\label{shonecomp}
Let $f:X\times \CCI \rightarrow X\times \CCI $ be a semi-hyperbolic 
polynomial skew product 
over $g:X\rightarrow X.$ 
 Suppose that for each $x\in X$, $d(x)\geq 2$, and that 
$\pc (P(f))\cap \CC $ is bounded in $\CC .$   
Let $\omega \in X$ be a point. 
If {\em int}$(K_{\omega }(f))$ is 
a non-empty connected set, then 
$J_{\omega }(f)$ is a Jordan curve.  
\end{prop}
\begin{proof}
By \cite[Theorem 1.12]{S4} and Lemma~\ref{fibconnlem}, we get that the 
unbounded component $A_{\omega }(f)$ of 
$F_{\omega }(f)$ is a John domain. Combining it,  
that $A_{\omega }(f)$ is simply connected (cf. Lemma~\ref{fibconnlem}), 
and \cite[page 26]{NV}, we see that 
$J_{\omega }(f)
=\partial (A_{\omega }(f))$ 
(cf. Lemma \ref{fibfundlem}) is locally connected. 
Moreover, by Lemma~\ref{fibfundlem}-\ref{fibfundlema}, 
we have 
$\partial (\mbox{int}(K_{\omega }(f)))=J_{\omega }(f).$ 
Hence, we see that $\CCI \setminus J_{\omega }(f)$ has 
exactly two connected components $A_{\g }(f)$ and int$(K_{\omega }(f))$, 
and that $J_{\omega }(f)$ is locally connected. From 
\cite[Lemma 5.1]{PT}, it follows that 
$J_{\g }(f)$ is a Jordan curve. Thus, we have proved 
Proposition~\ref{shonecomp}.
%
\end{proof}

\begin{lem}
\label{mainthranlem1}
Let $\G $ be a compact set in {\em Poly}$_{\deg \geq 2}.$ 
Let $f:\GN \times \CCI \rightarrow \GN \times \CCI $ be the  
skew product associated with the family $\G .$ 
Let $G$ be the polynomial semigroup generated by $\G .$ 
Suppose that $G\in {\cal G}$ and that $G$ is semi-hyperbolic. 
Moreover, suppose that there exist two elements 
$\alpha ,\beta \in \GN $ such that 
$J_{\beta }(f)<J_{\alpha }(f).$ 
Let $\g \in \GN $ and suppose that there exists a strictly increasing sequence 
$\{ n_{k}\} _{k\in \NN }$ of positive integers such that 
$\sigma ^{n_{k}}(\g )\rightarrow \alpha $ as $k\rightarrow \infty .$ 
Then, $J_{\g }(f)$ is a Jordan curve.  
\end{lem}
\begin{proof}
Since $G$ is semi-hyperbolic, \cite[Theorem 2.14-(4)]{S1} implies that 
\begin{equation}
\label{mainthranlem1eq1}
J_{\sigma ^{n_{k}}(\g )}(f)\rightarrow J_{\alpha }(f) 
\mbox{ as } k\rightarrow \infty, 
\end{equation}
with respect to the Hausdorff topology in the space of non-empty 
compact subsets of $\CCI .$ 
Combining it with Lemma~\ref{fiborder}, 
we see that there exists a number $k_{0}\in \NN $ such that 
for each $k\geq k_{0}$, 
\begin{equation}
\label{mainthranlem1eq2}
J_{\beta }(f)<J_{\sigma ^{n_{k}}(\g )}(f).
\end{equation}
We will show the following claim.\\ 
Claim: int$(K_{\g }(f))$ is connected.
 
 To show this claim, suppose that there exist two distinct components 
 $U_{1}$ and $U_{2}$ of 
 int$(K_{\g }(f)).$ 
Let $y_{i}\in U_{i}$ be a point, for each $i=1,2.$ 
Let $\epsilon >0$ be a number such that 
$\overline{D(K_{\beta }(f),\epsilon )}$ 
is included in a connected component $U$ of 
int$(K_{\alpha }(f)).$ 
Then, combining \cite[Theorem 2.14-(5)]{S1} and Lemma~\ref{constlimlem}, 
we get that there exists a number $k_{1}\in \NN $ with 
$k_{1}\geq k_{0}$ such that for each $k\geq k_{1}$ and each $i=1,2$, 
\begin{equation}
\label{mainthranlem1eq3}
f_{\g ,n_{k}}(y_{i})\in D(P^{\ast }(G),\epsilon )\subset 
\overline{D(K_{\beta }(f),\epsilon )}\subset U.
\end{equation}
Combining (\ref{mainthranlem1eq3}), (\ref{mainthranlem1eq1})
 and (\ref{mainthranlem1eq2}), 
we get that there exists a number $k_{2}\in \NN $ with 
$k_{2}\geq k_{1}$ such that for each $k\geq k_{2}$, 
\begin{equation}
\label{mainthranlem1eq4}
f_{\g ,n_{k}}(U_{1})=f_{\g ,n_{k}}(U_{2})=V_{k},
\end{equation}
where $V_{k}$ denotes the connected component of 
int$(K_{\sigma ^{n_{k}}(\g )}(f))$ containing 
$J_{\beta }(f).$ 
From (\ref{mainthranlem1eq2}) and (\ref{mainthranlem1eq4}), 
it follows that 
\begin{equation}
\label{mainthranlem1eq5}
(f_{\g ,n_{k}})^{-1}(J_{\beta }(f))
\subset \mbox{int}(K_{\g }(f)) 
\mbox{ and }
(f_{\g ,n_{k}})^{-1}(J_{\beta }(f))\cap 
U_{i}\neq \emptyset \ (i=1,2), 
\end{equation}
which implies that 
\begin{equation}
\label{mainthranlem1eq6}
(f_{\g ,n_{k}})^{-1}(J_{\beta }(f))
\mbox{ is disconnected.}
\end{equation}
For each $k\geq k_{2}$, let 
$\omega ^{k}:= (\g _{1},\ldots ,\g _{n_{k}},\beta _{1},\beta _{2},\ldots )
\in \GN .$ Then for each $k\geq k_{2}$, 
\begin{equation}
\label{mainthranlem1eq7}
(f_{\g ,n_{k}})^{-1}(J_{\beta }(f))=
J_{\omega ^{k}}(f).
\end{equation}
Since $G\in {\cal G}$, combining (\ref{mainthranlem1eq6}), 
(\ref{mainthranlem1eq7}) and Lemma~\ref{fibconnlem}
yields a contradiction. 
Hence, we have proved the claim.

 From the above claim and Proposition~\ref{shonecomp}, 
 it follows that 
$J_{\g }(f)$ is a Jordan curve.  
\end{proof}
We now investigate the situation 
that there exists a fiberwise Julia set which is a quasicircle 
and there exists another fiberwise Julia set which is not a Jordan curve.  
\begin{lem}
\label{mainthranlem2}
Let $\G $ be a non-empty compact subset of 
{\em Poly}$_{\deg \geq 2}.$ 
Let $f:\GN \times \CCI \rightarrow \GN \times \CCI $
 be the skew product associated with the family $\G $ of 
 polynomials. 
 Let $G$ be the polynomial semigroup generated by $\G .$ 
Let $\alpha ,\rho \in \GN $ be two elements. 
Suppose that $G\in {\cal G}$, that $G$ is semi-hyperbolic, 
that $\alpha $ is a periodic point of $\sigma :\GN \rightarrow \GN $, 
that $J_{\alpha }(f)$ is a quasicircle, 
and that $J_{\rho }(f)$ is not a Jordan curve. 
Then, for each $\epsilon >0$, there exist $n\in \NN $ and 
two elements $\theta _{1},\theta _{2}\in \partial D(0,1)$ 
with $\theta _{1}\neq \theta _{2}$ 
satisfying all of the following.
\begin{enumerate}
\item \label{mainthranlem2-1}
Let $\omega =(\alpha _{1},\ldots ,\alpha _{n},\rho _{1},\rho _{2},\ldots )
\in \GN $ and let $\psi :\CCI \setminus 
\overline{D(0,1)}\cong A_{\omega }(f)$ be a biholomorphic map 
with $\psi (\infty )=\infty .$ Moreover, 
for each $i=1,2$, let 
$T(\theta _{i}):=\psi (\{ r\theta _{i}\mid 1<r\leq \infty \} ).$ 
Then, there exists a point $p\in J_{\omega }(f)$ 
such that for each $i=1,2$, $T(\theta _{i})$ lands at $p.$ 
\item \label{mainthranlem2-2}
Let $V_{1}$ and $V_{2}$ be the two connected components of 
$\CCI \setminus (T(\theta _{1})\cup T(\theta _{2})\cup \{ p\} ).$ 
Then, for each $i=1,2$, 
$V_{i}\cap J_{\omega }(f)\neq \emptyset .$ 
Moreover, there exists an $i\in \{ 1,2\} $ such that 
{\em diam} $(V_{i}\cap K_{\omega }(f))\leq \epsilon $, 
and such that $V_{i}\cap J_{\omega }(f)\subset 
D(J_{\alpha }(f),\epsilon ).$  
\end{enumerate}   
\end{lem} 
\begin{proof}
For each $\g \in \GN $, let 
$\psi _{\g }:\CCI \setminus \overline{D(0,1)}\cong 
A_{\g }(f)$ be a biholomorphic map with 
$\psi _{\g }(\infty )=\infty .$ Moreover, 
for each $\theta \in \partial D(0,1)$, 
let 
$T_{\g }(\theta ):= \psi _{\g }(\{ r\theta \mid 1<r\leq \infty \} ).$ 
Since $G$ is semi-hyperbolic, 
combining \cite[Theorem 1.12]{S4}, 
Lemma~\ref{fibconnlem}, and \cite[page 26]{NV}, 
we see that for each $\g \in \GN $, 
$J_{\g }(f)$ is locally connected. 
Hence, for each $\g \in \GN $, 
$\psi _{\g }$ extends continuously over 
$\CCI \setminus D(0,1)$ such that 
$\psi _{\g }(\partial D(0,1))=J_{\g }(f).$ 
Moreover, since $G\in {\cal G}$, 
it is easy to see that for each $\g \in \GN $, 
there exists a number $a_{\g }\in \CC $ with $|a_{\g }|=1$ such that 
for each $z\in \CCI \setminus \overline{D(0,1)}$, 
we have $\psi _{\sigma (\g )}^{-1}\circ f_{\g ,1}\circ 
\psi _{\g }(z)=a_{\g }z^{d(\g )}.$

 Let $m\in \NN $ be an integer such that $\sigma ^{m}(\alpha )=
\alpha $ and let $h:= \alpha _{m}\circ \cdots \circ \alpha _{1}.$ 
Moreover, for each $n\in \NN $, 
we set $\omega ^{n}:= (\alpha _{1},\ldots ,\alpha _{mn},
\rho _{1},\rho _{2},\ldots )\in \GN .$ 
Then, $\omega ^{n}\rightarrow \alpha $ in $\GN $ as 
$n\rightarrow \infty .$ 
Combining it with \cite[Theorem 2.14-(4)]{S1}, 
we obtain  
\begin{equation}
\label{mainthranlem2eq1}
J_{\omega ^{n}}(f)\rightarrow 
J_{\alpha }(f) \mbox{ as } n\rightarrow \infty , 
\end{equation} 
with respect to the Hausdorff topology.
Let $\xi $ be a Jordan curve in int$(K(h))$ 
such that $P^{\ast }(\langle h\rangle )$ is included 
in the bounded component $B$ of $\CC \setminus \xi .$ 
By (\ref{mainthranlem2eq1}), 
there exists a $k\in \NN $ such that 
$J_{\omega ^{k}}(f)\cap (\xi \cup B) =\emptyset .$ 
We now show the following claim.\\ 
Claim 1: $\xi \subset $ int$(K_{\omega ^{k}}(f)).$ 

 To show this claim, suppose that $\xi $ is included in 
$A_{\omega ^{k}}(f)=
\CCI \setminus (K_{\omega ^{k}}(f)).$ 
Then, it implies that 
$f_{\omega ^{k},u}\rightarrow \infty $ on 
$P^{\ast }(\langle h\rangle )$ as $u\rightarrow \infty .$ 
 However, this is a 
contradiction, since $G\in {\cal G}.$ 
Hence, we have shown Claim 1.

 By Claim 1, we see that $P^{\ast }(\langle h\rangle )$ is 
 included in a bounded component $B_{0}$ of 
 int$(K_{\omega ^{k}}(f)).$ 
We now show the following claim.\\ 
Claim 2: $J_{\omega ^{k}}(f)$ is not a 
Jordan curve. 

 To show this claim, suppose that 
$J_{\omega ^{k}}(f)$ is a Jordan curve. Then, 
Lemma~\ref{060416lema} implies 
that $J_{\rho }(f)$ is a Jordan curve. 
However, this is a contradiction. Hence, we have shown Claim 2.

By Claim 2, 
there exist two distinct elements $t_{1},t_{2}\in \partial D(0,1)$ 
 and a point $p_{0}\in J_{\omega ^{k} }(f)$ 
 such that for each $i=1,2$, 
 $T_{\omega ^{k} }(t_{i})$ lands at the point $p_{0}.$ 
Let $W_{0}$ be the connected component of  
$\CCI \setminus (T_{\omega ^{k} }(t_{1})\cup T_{\omega ^{k}}(t_{2})\cup 
\{ p_{0}\} )$ such that 
$W_{0}$ does not contain $B_{0}.$ 
Then, we have 
\begin{equation}
\label{mainthranlem2eq3}
\overline{W_{0}}\cap P^{\ast }(\langle h\rangle )=\emptyset .
\end{equation}
For each $j\in \NN $, we take a connected component 
$W_{j}$ of $(h^{j})^{-1}(W_{0}).$ 
Then, $h^{j}:W_{j}\rightarrow W_{0}$ is biholomorphic.
We set $\zeta _{j}:= (h^{j}|_{W_{j}})^{-1} $ on 
$W_{0}.$ By (\ref{mainthranlem2eq3}), 
 there exists a number $R>0$ and a number $a >0$ 
 such that 
for each $j$, $\zeta _{j}$ is analytically continued to a univalent function 
$\tilde{\zeta _{j}}:B(\overline{W_{0}\cap D(0,R)}, a )\rightarrow 
\CCI $ and 
$W_{j}\cap (J_{\omega ^{k+j}}(f))\subset 
\tilde{\zeta _{j}}(W_{0}\cap  
D(0,R)).$ 
Hence, we obtain
\begin{equation}
\label{mainthranlem2eq4}
\mbox{diam }(W_{j}\cap K_{\omega ^{k+j}}(f))=
\mbox{diam }(W_{j}\cap J_{\omega ^{k+j}}(f))\rightarrow 
0 \mbox{ as } j\rightarrow \infty .
\end{equation}
Combining (\ref{mainthranlem2eq1}) and (\ref{mainthranlem2eq4}), 
there exists an $s\in \NN $
 such that 
diam $(W_{s}\cap K_{\omega ^{k+s}}(f))\leq \epsilon $, and such that 
$W_{s}\cap J_{\omega ^{k+s}}(f)\subset D(J_{\alpha }(f),\epsilon ).$ 

 Each connected component 
of $(\partial W_{s})\cap \CC $ is a connected component of 
\\ $(h^{s})^{-1}((T_{\omega ^{k}}(t_{1})\cup T_{\omega ^{k}}(t_{2})
\cup \{ p_{0}\} )\cap \CC )$, and there are some 
$u_{1},\ldots ,u_{v}\in \partial D(0,1)$ such that 
$\partial W_{s}=\overline{\bigcup _{i=1}^{v}T_{\omega ^{k+s}}(u_{i})}.$ 
Hence, $W_{s}$ is a Jordan domain. Therefore, 
$h^{s}:\overline{W_{s}}\rightarrow \overline{W_{0}}$ is a 
homeomorphism. 
Thus, $h^{s}:(\partial W_{s})\cap \CC \rightarrow 
(\partial W_{0})\cap \CC $ is a homeomorphism. 
Hence, $(\partial W_{s})\cap \CC $ is connected. 
It follows that there exist 
two elements $\theta _{1},\theta _{2}\in \partial D(0,1)$ 
with $\theta _{1}\neq \theta _{2}$  and 
a point $p\in J_{\omega ^{k+s}}(f)$ 
such that 
$\partial W_{s}=T_{\omega ^{k+s}}(\theta _{1})\cup 
T_{\omega ^{k+s}}(\theta _{2})\cup \{ p\} $, and such that 
for each $i=1,2$, $T_{\omega ^{k+s}}(\theta _{i})$ 
lands at the point $p.$ 
 By the argument of \cite[Lemma 17.5]{M}, 
each of two connected components of 
$\CCI \setminus 
( T_{\omega ^{k+s}}(\theta _{1})\cup 
T_{\omega ^{k+s}}(\theta _{2})\cup \{ p\} )$ intersects 
$J_{\omega ^{k+s}}(f).$ 

 Hence, we have proved Lemma~\ref{mainthranlem2}.
\end{proof}
\begin{lem}
\label{mainthranlem3}
Let $\G $ be a non-empty compact subset of 
{\em Poly}$_{\deg \geq 2}.$ 
Let $f:\GN \times \CCI \rightarrow \GN \times \CCI $
 be the skew product associated with the family $\G $ of 
 polynomials. 
 Let $G$ be the polynomial semigroup generated by $\G .$ 
Let $\alpha ,\beta ,\rho \in \GN $ be three elements. 
Suppose that $G\in {\cal G}$, that $G$ is semi-hyperbolic, 
that $\alpha $ is a periodic point of $\sigma :\GN \rightarrow \GN $, 
that $J_{\beta }(f)<J_{\alpha }(f)$, 
and that $J_{\rho }(f)$ is not a Jordan curve. 
Then, there exists an $n\in \NN $ such that 
setting $\omega := (\alpha _{1},\ldots ,\alpha _{n},
\rho _{1},\rho _{2},\ldots )\in \GN $ and 
${\cal U}:= \{ \g \in \GN 
\mid \exists \{ m_{j}\} _{j\in \NN }, \exists \{ n_{k}\} _{k\in \NN }, 
\sigma ^{m_{j}}(\g )\rightarrow \alpha ,\ 
\sigma ^{n_{k}}(\g )\rightarrow \omega \} $, 
we have that for each $\g \in {\cal U},$ 
$J_{\g }(f)$ is a Jordan curve but not a quasicircle,  
$A_{\g }(f)$ is a John domain, and the bounded component 
$U_{\g }$ of $F_{\g }(f)$ is not a John domain. 
\end{lem}
\begin{proof}

Let $p\in \NN $ be a number such that 
$\sigma ^{p}(\alpha )=\alpha $ and let 
$u :=\alpha _{p}\circ \cdots \circ \alpha _{1}.$ 
We show the following claim.\\ 
Claim 1: $J(u)$ is a quasicircle.

 To show this claim, by assumption,  
we have 
$J_{\beta }(f)<J(u).$ 
Let $\zeta := (\alpha _{1},\ldots $ $,\alpha _{p},\beta _{1},\beta _{2},\ldots 
)\in \GN .$ Then, 
we have $J_{\zeta }(f)=
u^{-1}(J_{\beta }(f)).$ Moreover, 
since $G\in {\cal G}$, we have that 
$J_{\zeta }(f)$ is connected. 
Hence, it follows that 
$u^{-1}(J_{\beta }(f))$ is connected. 
Let $U$ be a connected component of int$(K(u))$ containing 
$ J_{\beta }(f)$ and $V$ a connected component 
of int$(K(u))$ containing $u^{-1}(J_{\beta }(f)).$ 
By Lemma~\ref{fiborder}, it must hold that $U=V.$ 
Therefore, we obtain $u^{-1}(U)=U.$ Thus,  
int$(K(u))=U.$ Since $G$ is semi-hyperbolic, 
it follows that $J(u)$ is a quasicircle.  Hence, we have proved Claim 1.

 Let 
$\mu :=\frac{1}{3}\min 
 \{ |b-c|\mid b\in J_{\alpha }(f), c\in P^{\ast }(G)\} .$
Since $J_{\beta }(f)<J_{\alpha }(f)$, we have 
$P^{\ast }(G)\subset K_{\beta }(f).$ Hence, 
$\mu >0.$  
Applying Lemma~\ref{060416lemast} to the above 
$(f,\mu )$, let $\delta $ be the number 
in the statement of Lemma~\ref{060416lemast}. 
We set $\epsilon := \min \{ \delta ,\mu \} (>0).$
Applying Lemma~\ref{mainthranlem2} to 
the above $(\G ,\alpha , \rho , \epsilon )$, 
let $(n,\theta _{1},\theta _{2},\omega )$ be the 
element in the statement of Lemma~\ref{mainthranlem2}.
We set 
${\cal U}:= 
\{ \g \in \GN \mid 
\exists \{ m_{j}\} _{j\in \NN }, 
\exists \{ n_{k}\} _{k\in \NN }, 
\sigma ^{m_{j}}(\g )\rightarrow \alpha , 
\sigma ^{n_{k}}(\g )\rightarrow \omega \} .$ 
Then, combining 
the statement Lemma~\ref{060416lemast} and that of 
Lemma~\ref{mainthranlem2}, 
it follows that 
for any $\g \in {\cal U}$, 
$J_{\g }(f)$ is not a quasicircle. 
Moreover, by Lemma~\ref{mainthranlem1}, 
we see that 
for any $\g \in {\cal U}$, 
$J_{\g }(f)$ is a Jordan curve. 
Furthermore, combining the above argument, 
\cite[Theorem 1.12]{S4}, Lemma~\ref{fibconnlem}, 
and \cite[Theorem 9.3]{NV}, 
we see that for any $\g \in {\cal U}$, 
$A_{\g }(f)$ is a John domain, and 
the bounded component $U_{\g }$ of $F_{\g }(f)$ is not a John domain.
Therefore, we have proved Lemma~\ref{mainthranlem3}. 
\end{proof}
 We now demonstrate Theorem~\ref{mainthran1}.\\ 
{\bf Proof of Theorem~\ref{mainthran1}:} 
We suppose the assumption of Theorem~\ref{mainthran1}. 
We will consider several cases. 
First, we show the following claim.\\ 
Claim 1: If $J_{\g }(f)$ is a Jordan curve 
 for each $\g \in \GN $, then statement \ref{mainthran1-1} 
in Theorem~\ref{mainthran1} holds.

 To show this claim, 
 Lemma~\ref{060416lema} implies that  for each $\g \in X$, 
any critical point $v\in \pi ^{-1}(\{ \g \} )$ 
of $f_{\g }:\pi ^{-1}(\{ \g \} )\rightarrow 
\pi ^{-1}(\{ \sigma (\g )\} )$ 
(under the canonical identification $\pi ^{-1}(\{ \g \} )
\cong \pi ^{-1}(\{ \sigma (\g )\} )\cong \CCI $) 
belongs to 
$F^{\g }(f).$
Moreover, by \cite[Theorem 2.14-(2)]{S1}, 
$\tilde{J}(f)=\bigcup _{\g \in \GN }J^{\g }(f).$ 
Hence, it follows that 
$C(f)\subset \tilde{F}(f).$ Therefore, 
$C(f)$ is a compact subset of $\tilde{F}(f).$  
Since $f$ is semi-hyperbolic, 
\cite[Theorem 2.14-(5)]{S1} implies that 
$P(f)=\overline{\bigcup _{n\in \NN  }f^{n}(C(f))}\subset \tilde{F}(f).$ 
Hence, $f:\GN \times \CCI \rightarrow \GN \times  \CCI $ 
is hyperbolic. 
Combining it with Remark~\ref{hypskewsemigrrem}, 
we conclude that 
 $G$ is hyperbolic. 
  Moreover, Theorem~\ref{hypskewqc} implies that 
  there exists a constant $K\geq 1$ such that for each $\g \in \GN $, 
  $J_{\g }(f)$ is a $K$-quasicircle.  
Hence, we have proved Claim 1.

 Next, we will show the following claim.\\ 
Claim 2: If $J_{\alpha }(f)\cap J_{\beta }(f)
\neq \emptyset $ 
for each $(\alpha ,\beta )\in \GN \times \GN $, 
then $J(G)$ is arcwise connected. 

 To show this claim, since $G$ is semi-hyperbolic, 
combining \cite[Theorem 1.12]{S4},
Lemma~\ref{fibconnlem}, and \cite[page 26]{NV},   
we get that for each $\g \in \GN $, 
$A_{\g }(f)$ is a John domain and 
$J_{\g }(f)$ is locally connected. 
In particular, for each $\g \in \GN $, 
\begin{equation}
\label{mainthran1pfeq2}
J_{\g }(f) \mbox{ is arcwise connected.} 
\end{equation}
Moreover, by \cite[Theorem 2.14-(2)]{S1}, 
we have 
\begin{equation}
\label{mainthran1pfeq3}
\tilde{J}(f)=\bigcup _{\g \in \GN }J^{\g }(f).
\end{equation} 
 Combining (\ref{mainthran1pfeq2}), (\ref{mainthran1pfeq3}) and 
Lemma~\ref{fiblem}-\ref{pic}, 
we conclude that $J(G)$ is arcwise connected. Hence, we have proved 
Claim 2. 

  Next, we will show the following claim. \\ 
Claim 3: If $J_{\alpha }(f)\cap J_{\beta }(f)
\neq \emptyset $ 
for each $(\alpha ,\beta )\in \GN \times \GN $, and if 
there exists an element $\rho \in \GN $ such that 
$J_{\rho }(f)$ is not a Jordan curve, 
then statement \ref{mainthran1-3} in Theorem~\ref{mainthran1} holds. 

 To show this claim, let 
 ${\cal V}:= \bigcup _{n\in \NN }(\sigma ^{n})^{-1}(\{ \rho \} ).$ 
 Then, ${\cal V}$ is a dense subset of $\GN .$ 
From Lemma~\ref{060416lema}, it follows that 
 for each $\g \in {\cal V}$, $J_{\g }(f)$ is not a Jordan 
 curve. Combining this result with Claim 2, we conclude that statement 
 \ref{mainthran1-3} in Theorem~\ref{mainthran1} holds. Hence, 
 we have proved Claim 3. 

 We now show the following claim. \\ 
Claim 4: If there exist two elements $\alpha ,\beta \in \GN $ such that 
$J_{\alpha }(f)\cap J_{\beta }(f)= 
\emptyset $, and if there exists an element $\rho \in \GN $ such that 
$J_{\rho }(f)$ is not a Jordan curve, then 
statement \ref{mainthran1-2} in Theorem~\ref{mainthran1} holds.

 To show this claim, using Lemma~\ref{fiborder}, 
 We may assume that $J_{\beta }(f)<
 J_{\alpha }(f).$ 
Combining this, Lemma~\ref{fiborder}, 
\cite[Theorem 2.14-(4)]{S1}, and that 
the set of all periodic points of $\sigma $ in 
$\GN $ is dense in $\GN $, 
we may assume further that 
$\alpha $ is a periodic point of $\sigma .$ 
Applying Lemma~\ref{mainthranlem3} to 
$(\G, \alpha ,\beta , \rho )$ above, 
let $n\in \NN $ be the element 
in the statement of Lemma~\ref{mainthranlem3}, and  
we set  
$\omega =(\alpha _{1},\ldots ,\alpha _{n}, \rho _{1},\rho _{2},\ldots 
)\in \GN $ and 
${\cal U}:= 
\{ \g \in \GN \mid 
\exists (m_{j}), \exists (n_{k}), 
\sigma ^{m_{j}}(\g )\rightarrow \alpha ,
\sigma ^{n_{k}}(\g )\rightarrow \omega \} .$ 
Then, by the statement of Lemma~\ref{mainthranlem3}, 
we have that for each $\g \in {\cal U}$, 
$J_{\g }(f)$ is a Jordan curve but not a quasicircle, 
$A_{\g }(f)$ is a John domain, and the 
bounded component $U_{\g }$ of $F_{\g }(f)$ is not 
a John domain.   
Moreover, ${\cal U} $ is residual in $\GN $, and 
for any Borel probability measure 
$\tau $ on Poly$_{\deg \geq 2}$ with $\G _{\tau }=\G$, 
we have $\tilde{\tau }({\cal U})=1.$ 
Furthermore, let ${\cal 
V}:= \bigcup _{n\in \NN }(\sigma ^{n})^{-1}(\{ \rho \} ).$ Then, 
${\cal V}$ is a dense subset of $\GN $, 
and the argument in the proof 
of Claim 3 implies that for each $\g \in {\cal V}$, 
$J_{\g }(f)$ is not a Jordan curve. 
Hence, we have proved Claim 4.

 Combining Claims 1,2,3 and 4, Theorem~\ref{mainthran1} follows.    
\qed 

We now demonstrate Corollary~\ref{rancor2}.\\ 
{\bf Proof of Corollary~\ref{rancor2}:} 
From Theorem~\ref{mainthran1}, 
Corollary~\ref{rancor2} immediately follows.  
\qed 

\ 

 To demonstrate Theorem~\ref{mainthran2}, we need several lemmas.

\noindent {\bf Notation:} 
For a subset $A$ of $\CCI $, we denote by ${\cal C}(A)$ the set of 
all connected components of $A.$ 
\begin{lem}
\label{mainthran2lem1}
Let $f:X\times \CCI \rightarrow X\times \CCI $ be a polynomial 
skew product over $g:X\rightarrow X$ such that for each 
$x\in X$, $d(x)\geq 2.$ Let $\alpha \in X$ be a point. 
Suppose that $2\leq \sharp \left( {\cal C}(  
\mbox{{\em int}}(K_{\alpha }(f)))\right)<\infty .$ 
Then, 
$\sharp \left( {\cal C}(  
\mbox{{\em int}}(K_{g(\alpha )}(f)))\right)
<\sharp \left( {\cal C}(  
\mbox{{\em int}}(K_{\alpha }(f)))\right).$ 
In particular, there exists an $n\in \NN $ such that 
{\em int}$(K_{g^{n}(\alpha )}(f))$ is a non-empty connected set.
\end{lem} 
\begin{proof}
Suppose that 
$2\leq \sharp ({\cal C}(\mbox{int}(K_{g(\alpha )}(f))))=
\sharp ({\cal C}(\mbox{int}(K_{\alpha }(f))))<\infty .$ We 
will deduce a contradiction. 
Let $\{ V_{j}\} _{j=1}^{r}={\cal C}(\mbox{int}(K_{\alpha }(f)))$, 
where $2\leq r=\sharp ({\cal C}(\mbox{int}(K_{\alpha }(f))))<\infty .$  
Then, by the assumption above, 
we have that 
${\cal C}(\mbox{int}(K_{g(\alpha )}(f)))
=\{ f_{\alpha ,1}(V_{j})\} _{j=1}^{r}.$ 
For each $j=1,\ldots ,r$, let $p_{j}$ be the 
number of critical points of 
$f_{\alpha ,1}:V_{j}\rightarrow 
f_{\alpha ,1}(V_{j})$ counting multiplicities. 
Then, by the Riemann-Hurwitz formula, 
we have that for each $j=1,\ldots ,r$, 
$\chi (V_{j})+p_{j}=d\chi (f_{\alpha ,1}(V_{j}))$, 
where $\chi (\cdot) $ denotes the Euler number and 
$d:=\deg (f_{\alpha ,1}).$ 
Since $\chi (V_{j})=\chi (f_{\alpha ,1}(V_{j}))=1$ for each $j$, 
we obtain $r+\sum _{j=1}^{r}p_{j}=rd.$ 
Since $\sum _{j=1}^{r}p_{j}\leq d-1$, 
it follows that $rd-r\leq d-1.$ 
Therefore, we obtain $r\leq 1$, which is a contradiction. 
Thus, we have proved Lemma~\ref{mainthran2lem1}.
\end{proof}
\begin{lem}
\label{mainthran2lem2}
Let $f:X\times \CCI \rightarrow X\times \CCI $ be a 
polynomial skew product over $g:X\rightarrow X$ such that 
for each $x\in X$, $d(x)\geq 2.$ Let 
$\omega \in X$ be a point. 
Suppose that $f$ is hyperbolic, 
that $\pi _{\CCI }(P(f))\cap \CC $ is bounded in $\CC $, 
and that {\em int}$(K_{\omega }(f))$ is not connected. Then, 
there exist infinitely many connected components of 
{\em int}$(K_{\omega }(f)).$ 
\end{lem} 
\begin{proof}
Suppose that 
$2\leq \sharp ({\cal C}(\mbox{int}(K_{\omega }(f))))<\infty .$ 
Then, by Lemma~\ref{mainthran2lem1}, 
there exists an $n\in \NN $ such that 
int$(K_{g^{n}(\omega )}(f))$ is connected. 
We set $U:=$ int$(K_{g^{n}(\omega )}(f)).$ 
Let $\{ V_{j}\} _{j=1}^{r}$ be the set of 
all connected components 
of $(f_{\omega ,n})^{-1}(U).$ Since 
int$(K_{\omega }(f))$ is not connected, 
we have $r\geq 2.$   
For each $j=1,\ldots ,r$, we set 
$d_{j}:=\deg (f_{\omega ,n}:V_{j}\rightarrow 
U).$ Moreover, 
we denote by $p_{j}$ the number of critical points 
of $f_{\omega ,n}:V_{j}\rightarrow U$ 
counting multiplicities. 
Then, by the Riemann-Hurwitz formula, we see that for each 
$j=1,\ldots ,r$, 
$\chi (V_{j})+p_{j}=d_{j}\chi (U).$
Since $\chi (V_{j})=\chi (U)=1$ for each $j=1,\ldots ,r$, 
it follows that 
\begin{equation}
\label{mainthran2lem2eq2}
r+\sum _{j=1}^{r}p_{j}=d,
\end{equation}
where $d:= \deg (f_{\omega ,n}).$ 
Since $f$ is hyperbolic and $\pi _{\CCI }(P(f))\cap \CC $ is bounded in 
$\CC $, we have $\sum _{j=1}^{r}p_{j}=d-1.$ Combining it with 
(\ref{mainthran2lem2eq2}), we obtain $r=1$, 
which is a contradiction. Hence, we have proved Lemma~\ref{mainthran2lem2}. 
\end{proof}
\begin{lem}
\label{mainthran2lem3}
Let 
$f:X \times \CCI \rightarrow X \times \CCI $ be a polynomial skew 
product over $g:X\rightarrow X.$ 
Let 
$\alpha \in X $ be an element.  
Suppose that $\pi _{\CCI }(P(f))\cap \CC $ is bounded in $\CC $, that $f$ is hyperbolic, and 
that {\em int}$(K_{\alpha }(f)))$ is connected. Then, there exists a 
neighborhood ${\cal U}_{0}$ of $\alpha $ in $X$ satisfying the 
following.
\begin{itemize}
\item Let $\g \in X $ and suppose that there exists a sequence 
$\{ m_{j}\} _{j\in \NN }\subset \NN , m_{j}\rightarrow \infty $ such that 
for each $j\in \NN $, $g ^{m_{j}}(\g )\in {\cal U}_{0}.$ Then, 
$J_{\g }(f)$ is a Jordan curve.  
\end{itemize}  
\end{lem}
\begin{proof}
Let $P^{\ast }(f):= P(f)\setminus \pi _{\CCI }^{-1}(\{ \infty \} ).$ 
By assumption, we have 
$\pi _{\CCI }(P^{\ast }(f)\cap \pi ^{-1}(\{ \alpha \} ))\subset \mbox{ int}(K_{\alpha }(f)).$ 
Since int$(K_{\alpha }(f))$ is simply connected, 
there exists a Jordan curve $\xi $ in 
int$(K_{\alpha }(f))$ such that 
$\pi _{\CCI }(P^{\ast }(f)\cap \pi ^{-1}(\{ \alpha \} ))$ is included in the bounded component $B$ of 
$\CC \setminus \xi .$ 
Since $f$ is hyperbolic, \cite[Theorem 2.14-(4)]{S1} implies that 
the map $x \mapsto J_{x }(f)$ is continuous with respect to 
the Hausdorff topology. 
Hence, there exists a neighborhood ${\cal U}_{0}$ of $\alpha $ 
in $X $ such that for each $\beta \in {\cal U}_{0}$, 
$J_{\beta }(f)\cap (\xi \cup B)=\emptyset .$ Moreover, since $P(f) $ is compact, 
shrinking ${\cal U}_{0}$ if necessary, we may assume that for each $\beta \in {\cal U}_{0}$, 
$\pi _{\CCI }(P^{\ast }(f)\cap \pi ^{-1}(\{ \beta \} ))\subset B.$   
Since $\pi _{\CCI }(P(f))\cap \CC $ is bounded in $\CC $, it follows that for each $\beta \in {\cal U}_{0}$, 
$\xi <J_{\beta }(f).$ Hence, for each $\beta \in {\cal U}_{0}$, 
there exists a connected component $V_{\beta }$ of 
int$(K_{\beta }(f))$ such that 
\begin{equation}
\label{mainthran2lem3eq1}
\pi _{\CCI }(P^{\ast }(f)\cap \pi ^{-1}(\{ \beta \} ))\subset V_{\beta }.
\end{equation} 
Let $\g \in X$ be an element and suppose that 
there exists a sequence $\{ m_{j}\} _{j\in \NN }\subset \NN , 
m_{j}\rightarrow \infty $
 such that for each $j\in \NN $, $g ^{m_{j}}(\g )\in 
 {\cal U}_{0}.$ 
We will show that int$(K_{\g }(f))$ is connected. 
Suppose that there exist two distinct connected components  
$W_{1}$ and $W_{2}$ of int$(K_{\g }(f)).$ 
Then, combining 
\cite[Corollary 2.7]{S4} and 
(\ref{mainthran2lem3eq1}), we get that 
there exists a $j\in \NN $ such that 
\begin{equation}
\label{mainthran2lem3eq2}
\pi _{\CCI }(P^{\ast }(f)\cap \pi ^{-1}(\{ \beta \} ))\subset f_{\g ,m_{j}}(W_{1})=f_{\g ,m_{j}}(W_{2}).
\end{equation}
We set $W=f_{\g ,m_{j}}(W_{1})=f_{\g ,m_{j}}(W_{2}).$ 
Let $\{ V_{i}\} _{i=1}^{r}$ be the set of all connected components of 
$(f_{\g ,m_{j}})^{-1}(W).$ 
Since $W_{1}\neq W_{2}$, we have $r\geq 2.$ 
For each $i=1,\ldots ,r$, we denote by $p_{i}$ the 
number of critical points of $f_{\g ,m_{j}}:V_{i}\rightarrow W$ 
counting multiplicities. Moreover, we set 
$d_{i}:= \deg (f_{\g ,m_{j}}:V_{i}\rightarrow W).$ 
Then, by the Riemann-Hurwitz formula, we see that for each 
$i=1,\ldots ,r$, 
$\chi (V_{i})+p_{i}=d_{i}\chi (W).$ 
Since $\chi (V_{i})=\chi (W)=1$, it follows that 
\begin{equation}
\label{mainthran2lem3eq3}
r+\sum _{i=1}^{r}p_{i}=d, \mbox{ where }d:= \deg (f_{\g ,m_{j}}).
\end{equation} 
By (\ref{mainthran2lem3eq2}), we have $\sum _{i=1}^{r}p_{i}=d-1.$ 
Hence, (\ref{mainthran2lem3eq3}) implies 
$r=1$, which is a contradiction. Therefore, 
int$(K_{\g }(f))$ is a non-empty connected set. 
Combining it with Proposition~\ref{shonecomp}, 
we conclude that $J_{\g }(f)$ is a Jordan curve.   

 Thus, we have proved Lemma~\ref{mainthran2lem3}.
\end{proof}
We now demonstrate Theorem~\ref{mainthran2}.\\ 
{\bf Proof of Theorem~\ref{mainthran2}:} 
We suppose the assumption of Theorem~\ref{mainthran2}. 
We consider the following three cases. 

 \noindent Case 1: For each $\g \in \GN $, int$(K_{\g }(f))$ is connected. \\ 
 Case 2: For each $\g \in \GN $, int$(K_{\g }(f))$ is disconnected.\\ 
 Case 3: There exist two elements $\alpha \in \GN $ and $\beta \in \GN $ 
 such that int$(K_{\alpha }(f))$ is connected and such that 
 int$(K_{\beta }(f))$ is 
 disconnected. 

 Suppose that we have Case 1. 
 Then, by Theorem~\ref{hypskewqc}, there exists a constant $K\geq 1$ such that 
 for each $\g \in \GN $, $J_{\g }(f)$ is a $K$-quasicircle. 

 Suppose that we have  Case 2. Then, by 
 Lemma~\ref{mainthran2lem2}, we get that 
 for each $\g \in \GN $, there exist infinitely many 
 connected components of int$(K_{\g }(f)).$ 
 Moreover, by Theorem~\ref{mainthran1}, 
 we see that statement \ref{mainthran1-3} in Theorem~\ref{mainthran1} 
 holds. Hence, statement \ref{mainthran2-3} in Theorem~\ref{mainthran2} holds.

 Suppose that we have Case 3. 
 By Lemma~\ref{mainthran2lem2}, there exist infinitely many 
 connected components of int$(K_{\beta }(f)).$ 
 Let ${\cal W}:= \bigcup _{n\in \NN }(\sigma ^{n})^{-1}(\{ \beta \} ).$ 
 Then, for each $\g \in {\cal W}$, there exist infinitely many 
 connected components of int$(K_{\g }(f)).$ Moreover, 
 ${\cal W}$ is dense in $\GN .$ 
 
Next, combining Lemma~\ref{mainthran2lem3} and that the set of all periodic points of $\sigma :\GN \rightarrow 
  \GN $ is dense in $\GN $, we may assume that 
  the above $\alpha $ is a periodic point of $\sigma .$ 
  Then, $J_{\alpha }(f)$ is a quasicircle. 
 We set ${\cal V}:= \bigcup _{n\in \NN }(\sigma ^{n})^{-1}(\{ \alpha \} ).$ 
 Then ${\cal V}$ is dense in $\GN .$ 
 Let $\g \in {\cal V}$ be an element. 
Then there exists an $n\in \NN $ such that 
$\sigma ^{n}(\g )=\alpha .$ 
Since $(f_{\g ,n})^{-1}(K_{\alpha }(f))=
K_{\g }(f)$, it follows that 
$\sharp ({\cal C}($int$(K_{\g }(f))))<\infty .$ 
Combining it with Lemma~\ref{mainthran2lem2} 
and Proposition~\ref{shonecomp}, 
we get that $J_{\g }(f)$ is a Jordan curve. 
Combining it with 
 that $J_{\alpha }(f)$ is a quasicircle, 
 it follows that $J_{\g }(f)$ is a quasicircle. 

Next, let 
$\mu := \frac{1}{3}
\min \{ |b-c| \mid b\in J(G),\ c\in P^{\ast }(G)\} (>0).$ 
Applying Lemma~\ref{060416lemast} to 
$(f, \mu )$ above, 
let $\delta $ be the number in the statement of 
Lemma~\ref{060416lemast}. 
We set $\epsilon := \min \{ \delta ,\mu \} $ and 
$\rho := \beta .$  
Applying Lemma~\ref{mainthranlem2} to $(\G ,\alpha ,\rho ,\epsilon )$ 
above, let 
$(n,\theta _{1},\theta _{2},\omega )$ be the element 
in the statement of Lemma~\ref{mainthranlem2}. 
Let 
${\cal U}:= 
\{ \g \in \GN \mid \exists \{ m_{j}\} _{j\in \NN }, 
\exists \{ n_{k}\} _{k\in \NN },  
\sigma ^{m_{j}}(\g )\rightarrow \alpha , 
\sigma ^{n_{k}}(\g )\rightarrow \omega \} .$ 
Then, combining the statement of Lemma~\ref{060416lemast} 
and that of Lemma~\ref{mainthranlem2}, 
it follows that for any $\g \in {\cal U}$, 
$J_{\g }(f)$ is not a quasicircle. 
Moreover, by Lemma~\ref{mainthran2lem3}, 
we get that for any $\g \in {\cal U}$, 
$J_{\g }(f)$ is a Jordan curve. 
Combining the above argument, \cite[Theorem 1.12]{S4}, 
Lemma~\ref{fibconnlem}, and \cite[Theorem 9.3]{NV}, 
we see that for any $\g \in {\cal U}$, 
$A_{\g }(f)$ is a John domain, and the bounded component 
$U_{\g }$ of $F_{\g }(f)$ is not a John domain. 
Furthermore, 
it is easy to see that 
${\cal U}$ is residual in $\GN $, and that 
for any 
Borel probability measure $\tau $ on 
Poly$_{\deg \geq 2}$ with $\G _{\tau }=\G $, 
$\tilde{\tau }({\cal U})=1.$ 
Thus, we have proved Theorem~\ref{mainthran2}. 
%
%
\qed 
\begin{rem}
Using the above method (especially, using 
Lemmas~\ref{060416lemast} and \ref{mainthran2lem3}), 
we can also construct an example of a polynomial skew product $f:\CC ^{2}\rightarrow \CC ^{2}, 
f(z,w)=(p(z), q_{z}(w))$, where $p:\CC \rightarrow \CC $ is a polynomial with $\deg (p)\geq 2$, $q_{z}: \CC \rightarrow 
\CC $ is a monic polynomial with $\deg (q_{z})\geq 2$ for each $z\in \CC $, and $(z,w)\rightarrow q_{z}(w)$ is a polynomial of 
$(z,w)$,  
such that all of the following hold: 
\begin{itemize}
\item[(I)] $f:\CC^{2}\rightarrow \CC^{2}$ satisfies the Axiom A; 
\item[(II)] for each $z\in J(p)$, the fiberwise Julia set $J_{z}(f)$ is connected; and  
\item[(III)] for almost every $z\in J(p)$ with respect to the maximal entropy measure of $p:\CC \rightarrow \CC $, 
the fiberwise Julia set $J_{z}(f)$ is a Jordan curve but not a quasicircle,  the fiberwise basin 
$A_{z}(f)$ of $\infty $ is a John domain, and the bounded component of $F_{z}(f)$ is not a John domain.  
\end{itemize}
More precisely, for any $R,\epsilon >0$ and $n\in \NN $, 
let $p_{R}(z):=z^{2}-R, p:= p_{R}^{n}, h_{\epsilon }(w):= (w-\epsilon )^{2}-1+\epsilon $ and define 
$t_{n,\epsilon }(w)$ by $h_{\epsilon }^{n}(w)=w^{2^{n}}+t_{n,\epsilon }(w).$ 
For appropriate choice of $\epsilon $ small, and $R,n$ large with $n$ even, 
the map 
$f(z,w)=(p(z), w^{2^{n}}+\frac{z+\sqrt{R}}{2\sqrt{R}}t_{n,\epsilon }(w))$ satisfies 
(I)(II)(III) above. 

To explain the proof, note that 
$J(p)$ is contained in the union of two disks $D=D(\sqrt{R},r)$ and $-D$, 
for an $r>0$ such that $r/\sqrt{R}\rightarrow 0$ as $R\rightarrow \infty .$ 
Let $\epsilon >0$  be small. Let $n$ be large and let  $g(w)=w^{2^{n}}.$ 
Then there exists an open disk $B_{1}$ around $-1+\epsilon $ and an open disk $B_{2}$ with 
$\{ 0,\epsilon \} \subset B_{2}$ such that 
$h_{\epsilon }^{n}(B_{1})$ is a relative compact subset of $B_{1}$, 
$h_{\epsilon }^{n}(B_{2})$ is a relative compact subset of $B_{2}$, 
and $g(B_{1}\cup B_{2})$ is a relative compact subset of $B_{2}.$ Let $R$ be so large. Then 
there exists a compact subset $B_{1}'$ of $B_{1}$ and a compact subset $B_{2}'$ of $B_{2}$ such that 
\begin{itemize}
\item[(i)] for each $z\in D$, $q_{z}(B_{1})\subset B_{1}'$ and $q_{z}(B_{2})\subset B_{2}'$ and each 
finite critical value of $q_{z}$ is included in $B_{1}'\cup B_{2}'$, and 
\item[(ii)] for each $z\in -D$, $q_{z}(B_{1}\cup B_{2})\subset B_{2}'$ and each finite critical value of $q_{z}$ is included 
in $B_{2}'.$ 
\end{itemize}
From (i) (ii) and Lemmas~\ref{060416lemast} and \ref{mainthran2lem3}, it is easy to see that $f:\CC^{2}\rightarrow \CC^{2}$ satisfies (I)(II)(III) above.  

This example from the author of this paper has been announced in \cite[Example 5.10]{DH} as ``Sumi's example.'' 
For the related topics of Axiom A polynomial skew products on $\CC ^{2}$, see \cite{DH}. 
Note that statement (2) in \cite[Example 5.10]{DH}, which was added by the authors of \cite{DH} to the original 
example from the author of this paper, is unfortunately false. 
More precisely, the author of this paper found that the above example gives a 
counterexample to \cite[Lemma 3.5, Theorem 5.2, Corollary 5.3]{DH}. This matter will be reported in \cite{DH2}. 
\end{rem}

We now demonstrate Proposition~\ref{ranprop1}.\\ 
{\bf Proof of Proposition~\ref{ranprop1}:}
Since $P^{\ast }(G)\subset $ int$(\hat{K}(G))\subset F(G)$, 
$G$ is hyperbolic.  
Let $\g \in \GN $ be any element. We will show the following claim.\\ 
Claim: int$(K_{\g }(f))$ is a non-empty connected set. 

 To show this claim, since $G$ is hyperbolic, int$(K_{\g }(f))$ is non-empty. 
Suppose that there exist two distinct connected components 
$W_{1}$ and $W_{2}$ of int$(K_{\g }(f)).$ 
Since $P^{\ast }(G)$ is included in a 
connected component $U$ of int$(\hat{K}(G))$ $
\subset F(G)$, 
\cite[Corollary 2.7]{S4} implies that 
there exists an $n\in \NN $ such that 
$P^{\ast }(G)\subset f_{\g ,n}(W_{1})=f_{\g ,n}(W_{2}).$ 
Let $W:= f_{\g ,n}(W_{1})=f_{\g ,n}(W_{2}).$ Then, 
any critical value of $f_{\g ,n}$ in $\CC $ is included in  $W.$ 
Using the method in the proof of Lemma~\ref{mainthran2lem3}, 
we see that $(f_{\g ,n})^{-1}(W)$ is connected. However, 
this is a contradiction, since $W_{1}\neq W_{2}.$ Hence, we have 
proved the above claim. 

 From Claim above and Theorem~\ref{hypskewqc}, 
 it follows that there exists a constant $K\geq 1 $ such that 
 for each $\g \in \GN $, 
 $J_{\g }(f)$ is a $K$-quasicircle. 

 Hence, we have proved Proposition~\ref{ranprop1}. 
\qed

\section{Construction of examples}
\label{Const}
We present a way 
to construct examples of semigroups $G$ in ${\cal G}_{dis}.$ 
\begin{lem}[\cite{SdpbpI}]
\label{Constprop}
Let $G$ be a 
polynomial semigroup generated by 
a compact subset $\G $ of {\em Poly}$_{\deg \geq 2}.$ 
Suppose that $G\in {\cal G}$ and  
{\em int}$(\hat{K}(G))\neq \emptyset .$ 
Let $b\in $ {\em int}$(\hat{K}(G)).$ 
Moreover, let $d\in \NN $ be any positive integer such that 
$d\geq 2$, and such that 
$(d, \deg (h))\neq (2,2)$ for each $h\in \G .$ 
Then, there exists a number $c>0$ such that 
for each $a\in \CC $ with $0<|a|<c$, 
there exists a compact neighborhood $V$ of 
$g_{a}(z)=a(z-b)^{d}+b$ in {\em Poly}$_{\deg \geq 2}$ satisfying 
that for any non-empty subset $V'$ of $V$,  
the polynomial semigroup 
$\langle \G \cup  V'\rangle $ generated by the family $\G \cup V'$ 
belongs to ${\cal G}_{dis}$ and $\hat{K}(\langle \G\cup V'\rangle)=\hat{K}(G).$  
Moreover, in addition to the assumption above, 
if $G$ is semi-hyperbolic (resp. hyperbolic), 
then the above $\langle \G \cup V'\rangle $ is semi-hyperbolic (resp. hyperbolic).   
\end{lem}
\begin{proof}
We follow the proof in \cite{SdpbpI}. 
Conjugating $G$ by $z\mapsto z+b$, we may assume that $b=0.$ 
For each $h\in \G$, let $a_{h}$ be the coefficient of 
the highest degree term of $h$  
and let $d_{h}:=\deg (h).$ Let $r>0$ be a number 
such that $\overline{D(0,r)}\subset 
$ int$(\hat{K}(G)).$ 

Let $h\in \G $ and let $\alpha >0$ be a number. 
Since $d\geq 2$ and $(d,d_{h})\neq (2,2)$, 
it is easy to see that 
$(\frac{r}{\alpha })^{\frac{1}{d}}>
2\left(\frac{2}{|a_{h}|}(\frac{1}{\alpha })
^{\frac{1}{d-1}}\right)^{\frac{1}{d_{h}}}
$ if and only if 
\begin{equation}
\label{Contproppfeq1}
\log \alpha <
\frac{d(d-1)d_{h}}{d+d_{h}-d_{h}d}
( \log 2-\frac{1}{d_{h}}\log \frac{|a_{h}|}{2}-\frac{1}{d}\log r) .
\end{equation} 
We set 
\begin{equation}
\label{Contproppfeq2}
c_{0}:=\min _{h\in \G }\exp \left(\frac{d(d-1)d_{h}}{d+d_{h}-d_{h}d}
( \log 2-\frac{1}{d_{h}}\log \frac{|a_{h}|}{2}-\frac{1}{d}\log r) \right)
\in (0,\infty ).
\end{equation}
Let $0<c<c_{0}$ be a small number and let $a\in \CC $ 
be a number with $0<|a|<c.$ 
Let $g_{a}(z)=az^{d}.$ 
Then, we obtain $K(g_{a})=\{ z\in \CC \mid 
|z|\leq (\frac{1}{|a|})^{\frac{1}{d-1}}\} $ and 
$g_{a}^{-1}(\{ z\in \CC \mid  |z|=r\} )=
\{ z\in \CC \mid |z|=(\frac{r}{|a|})^{\frac{1}{d}}\} .$
Let 
$D_{a}:=\overline{D(0,2(\frac{1}{|a|})^{\frac{1}{d-1}})}.$ 
Since $h(z)=a_{h}z^{d_{h}}(1+o(1))\ (z\rightarrow \infty )$ uniformly on 
$\G $, 
it follows that if $c$ is small enough, then 
for any $a\in \CC $ with $0<|a|<c$ and for any $h\in \G $, 
$h^{-1}(D_{a})\subset 
\left\{ z\in \CC \mid 
|z|\leq 2\left( \frac{2}{|a_{h}|}(\frac{1}{|a|})^{\frac{1}{d-1}}\right) 
^{\frac{1}{d_{h}}}\right\} .$  
This implies that for each $h\in \G $, 
\begin{equation}
\label{Contproppfeq3}
h^{-1}(D_{a})\subset g_{a}^{-1}(\{ z\in \CC \mid |z|<r\} ).
\end{equation} 
Moreover, if $c$ is small enough, then for any $a\in \CC $ with 
$0<|a|<c$ and any $h\in \G $,  
\begin{equation}
\label{Contproppfeq4}
\hat{K}(G)\subset g_{a}^{-1}(\{ z\in \CC \mid |z|<r\} ),\ 
\overline{h(\CCI \setminus D_{a})}\subset 
\CCI \setminus D_{a}.
\end{equation}
Let $a\in \CC $ with $0<|a|<c.$  
By (\ref{Contproppfeq3}) and (\ref{Contproppfeq4}), 
there exists a compact neighborhood $V$ of $g_{a}$ in Poly$_{\deg \geq 2}$,  
such that 
\begin{equation}
\label{Contproppfeq5} 
\hat{K}(G)\cup \bigcup _{h\in \G }h^{-1}
(D_{a})
\subset 
\mbox{int}\left( \bigcap _{g\in V}
g^{-1}(\{ z\in \CC \mid |z|<r\} )\right), \mbox{ and } 
\end{equation}
\begin{equation}
\label{Contproppfeq5-a}
\bigcup _{h\in \G \cup V}
\overline{h(\CCI \setminus D_{a})}
\subset \CCI \setminus D_{a},
\end{equation}
which implies that 
\begin{equation}
\label{Contproppfeq5-1}
\mbox{int}(\hat{K}(G))\cup 
(\CCI \setminus D_{a})
\subset F(\langle \G \cup V\rangle ). 
\end{equation}

 By (\ref{Contproppfeq5}), we obtain that for any non-empty subset 
 $V'$ of $V$, 
\begin{equation}
\label{Contproppfeq7}
\hat{K}(G)=\hat{K}(\langle \G \cup  V'\rangle  ). 
\end{equation} 
If the compact neighborhood $V$ of $g_{a}$ is so small, then  
\begin{equation}
\label{Contproppfeq8}
\bigcup _{g\in V} CV^{\ast }(g)  \subset  \mbox{int}(\hat{K}(G)).
\end{equation} 
Since $P^{\ast }(G)\subset \hat{K}(G)$, 
combining it with (\ref{Contproppfeq7}) and (\ref{Contproppfeq8}),
we get  
that for any non-empty subset $V'$ of $V$, 
$P^{\ast }(\langle \G \cup V'\rangle  )\subset 
\hat{K}(\langle \G \cup V'\rangle  ).$ 
 Therefore, for any non-empty subset $V'$ of $V$,  
$\langle \G \cup V'\rangle  \in {\cal G}.$ 

We now show that for any non-empty subset $V'$ of $V$, $J(\langle \G \cup V'\rangle )$ is disconnected.  
Let $$U:=\left(\mbox{int}(\bigcap _{g\in V}g^{-1}(\{ z\in \CC \mid |z|<r\} ))\right)
\setminus 
\bigcup _{h\in \G }h^{-1}(D_{a}).$$ 
Then, for any $h\in \G $, 
\begin{equation}
\label{Contproppfeq6}
h(U)\subset \CCI \setminus D_{a}.
\end{equation} 
Moreover, for any $g\in V$, $g(U)\subset 
$ int$(\hat{K}(G)).$ 
Combining it with (\ref{Contproppfeq5-1}), 
(\ref{Contproppfeq6}), and Lemma~\ref{hmslem}-\ref{bss}, 
it follows that $U\subset 
F(\langle \G \cup V\rangle  ).$ 
If the neighborhood $V$ of $g_{a}$ is  so small, then 
there exists an annulus $A$ in $U$ such that for any $g\in V$,  
$A$ separates $J(g)$ and $\bigcup _{h\in \G }h^{-1}(J(g)).$  
Hence, it follows that for any non-empty subset $V'$ of $V$, 
the polynomial semigroup $\langle \G \cup  V'\rangle $ generated by 
the family $\G \cup V'$ satisfies that $J(\langle \G \cup V'\rangle )$ is 
disconnected.  

 We now suppose that in addition to the assumption, 
 $G$ is semi-hyperbolic. Let $V'$ be any non-empty subset of 
 $V .$ Since $G$ is semi-hyperbolic, 
$UH(G)\cap \CC \subset P^{\ast }(G)\cap F(G)\subset \hat{K}(G)\cap F(G)=\mbox{int}(\hat{K}(G)).$ 
Moreover, by (\ref{Contproppfeq5-1}), 
$J(\langle \G \cup  V'\rangle )\subset \CCI \setminus \left((\CCI \setminus D_{a})\cup \mbox{int}(\hat{K}(G))\right).$ 
Therefore, there exists a positive integer $N$ and a positive number $\delta $ 
such that for each $z\in J(\langle \G \cup  V'\rangle )$ and each $h\in G$, we have 
\begin{equation}
\label{eq:lesn}
\deg (h: W\rightarrow D(z,\delta ))\leq N
\end{equation}
for each connected component $W$ of $h^{-1}(D(z,\delta )).$ 
Since $P^{\ast }(\langle \G \cup  V'\rangle )\subset \hat{K}(\langle \G \cup V'\rangle )=\hat{K}(G)$, 
(\ref{Contproppfeq5}) 
implies that 
there exists a positive number $\delta _{1}$ such that for each 
$z\in \bigcup _{g\in V'}g^{-1}(J(\langle \G \cup V'\rangle))$ and each $\beta \in \langle \G \cup V'\rangle $, 
\begin{equation}
\label{eq:beta1}
\deg (\beta :B \rightarrow D(z,\delta _{1}))=1, 
\end{equation}
for each connected component $B$ of $\beta ^{-1}(D(z,\delta _{1})).$ 
By (\ref{Contproppfeq8}), there exists a positive number $\delta _{2}$ such that 
for each $z\in J(\langle \G \cup V'\rangle )$ and each $\alpha \in V'$, 
\begin{equation}
\label{eq:alpha1}
\mbox{diam}\ Q\leq \delta _{1}, \deg (\alpha :Q\rightarrow D(z,\delta _{2}))=1
\end{equation}
for each connected component $Q$ of $\alpha ^{-1}(D(z,\delta _{2})).$ 
Furthermore, by (\ref{eq:lesn}) and \cite[Lemma 1.10]{S1} (or \cite{S2}), 
there exists a constant
 $0<c<1$ such that for each $z\in J(\langle \Gamma \cup V'\rangle )$ and each $h\in G$, 
\begin{equation}
\label{eq:diamS}
\mbox{diam }S\leq \delta _{2}, 
\end{equation} 
for each connected component $S$ of $h^{-1}(D(z,c\delta )).$  
Let $\zeta \in \langle \G \cup V'\rangle $ be any element. 
If $\zeta \in G$, then by (\ref{eq:lesn}), 
for each $z\in J(\langle \G \cup V'\rangle )$, 
we have $\deg (\zeta : W\rightarrow D(z,c\delta ))\leq N$, 
for each connected component $W$ of $h^{-1}(D(z,c\delta )).$ 
If $\zeta $ is of the form $\zeta =h\circ \alpha \circ \beta $, where 
$h\in G\cup \{ Id\} $, $\alpha \in V'$, and $\beta \in \langle \G \cup V'\rangle \cup \{ Id\} $, 
then combining (\ref{eq:lesn}), (\ref{eq:beta1}), and (\ref{eq:alpha1}), 
we get that for each $z\in J(\langle \G \cup  V'\rangle )$, 
$\deg (\zeta :Q\rightarrow D(z,c\delta ))\leq N$ 
for each connected component $Q$ of $\zeta ^{-1}(D(z,c\delta )).$ 
Therefore, $J(\langle \G \cup V'\rangle )\subset SH_{N}(\langle \G \cup V'\rangle )$ and 
$\langle \G \cup V'\rangle $ is semi-hyperbolic. 

 We now suppose that in addition to the assumption, 
$G$ is hyperbolic.  Let $V'$ be any non-empty subset of 
$V.$ By the above argument with $N=1$, 
we obtain that $\langle \G \cup V'\rangle $ is hyperbolic. 

 Thus, we have proved Lemma~\ref{Constprop}. 
\end{proof}
\begin{lem}[\cite{SdpbpI}] 
\label{shshfinprop}
Let $m\geq 2$ and let $d_{2},\ldots ,d_{m}\in \NN $ be such that  
$d_{j}\geq 2$ for each $j=2,\ldots ,m.$ Let 
$h_{1}\in {\cal Y}_{d_{1}}$ with {\em int}$(K(h_{1}))\neq \emptyset $ be  
such that $\langle h_{1}\rangle \in {\cal G}.$ 
Let $b_{2},b_{3},\ldots ,b_{m}\in $ {\em int}$(K(h_{1})).$ 
Then, all of the following statements hold. 
\begin{enumerate}
\item \label{shshfinprop1}
Suppose that $\langle h_{1}\rangle $ is 
semi-hyperbolic (resp. hyperbolic). 
Then, there exists a number $c>0$ such that 
for each $(a_{2},a_{3},\ldots ,a_{m})\in \CC ^{m-1}$ with 
$0<|a_{j}|<c$ ($j=2,\ldots ,m$), 
setting $h_{j}(z)=a_{j}(z-b_{j})^{d_{j}}+b_{j}$ ($j=2,\ldots ,m$), 
the polynomial semigroup 
$G=\langle h_{1},\ldots ,h_{m}\rangle $ satisfies that 
$G\in {\cal G}$, $\hat{K}(G)=K(h_{1})$ and $G$ is semi-hyperbolic (resp. hyperbolic). 
\item \label{shshfinprop2}
Suppose that $\langle h_{1}\rangle $ is 
semi-hyperbolic (resp. hyperbolic). Suppose also that 
either (i) there exists a $j\geq 2$ with $d_{j}\geq 3$, or 
(ii) $\deg(h_{1})=3$, $b_{2}=\cdots =b_{m}.$ Then, there exist 
$a_{2},a_{3},\ldots ,a_{m}>0$ such that setting 
$h_{j}(z)=a_{j}(z-b_{j})^{d_{j}}+b_{j}$ ($j=2,\ldots ,m$), 
the polynomial semigroup $G=\langle h_{1},h_{2},\ldots ,h_{m}\rangle $ 
satisfies that $G\in {\cal G}_{dis}$, $\hat{K}(G)=K(h_{1})$ and 
$G$ is semi-hyperbolic (resp. hyperbolic).  

\end{enumerate} 
\end{lem}
\begin{proof} 
We will follow the proof in \cite{SdpbpI}. 
First, we show \ref{shshfinprop1}. 
Let $r>0$ be a number such that 
$D(b_{j},2r)\subset \mbox{int}(K(h_{1}))$ for each 
$j=1,\ldots ,m.$ 
If we take $c>0$ so small, then 
for each $(a_{2},\ldots ,a_{m})\in \CC ^{m-1}$ 
such that $0<|a_{j}|<c$ for each $j=2,\ldots ,m$, 
setting $h_{j}(z)=a_{j}(z-b_{j})^{d_{j}}+b_{j}$ 
($j=2,\ldots ,m$), we have 
\begin{equation}
\label{shshfinpropeq1}
h_{j}(K(h_{1}))\subset D(b_{j},r)\subset 
\mbox{int}(K(h_{1}))\ (j=2,\ldots ,m). 
\end{equation} 
Hence, $K(h_{1})=\hat{K}(G)$, 
where $G=\langle h_{1},\ldots ,h_{m}\rangle .$ 
Moreover, by (\ref{shshfinpropeq1}), 
we have $P^{\ast }(G)\subset K(h_{1}).$ 
Hence, $G\in {\cal G}.$ 

 If $\langle h_{1}\rangle $ is semi-hyperbolic, 
then using the same method as that 
in the proof of Lemma~\ref{Constprop}, 
we obtain that $G$ is semi-hyperbolic. 

 We now suppose that $\langle h_{1}\rangle $ is hyperbolic. 
By (\ref{shshfinpropeq1}), we have 
$\bigcup _{j=2}^{m}CV^{\ast }(h_{j})\subset 
\mbox{int}(\hat{K}(G)).$ Combining it with 
the same method as that in the proof of Lemma~\ref{Constprop}, 
we obtain that $G$ is hyperbolic. 
Hence, we have proved statement \ref{shshfinprop1}. 

 We now show statement \ref{shshfinprop2}. 
Suppose we have case (i). 
We may assume $d_{m}\geq 3.$ 
Then, by statement \ref{shshfinprop1}, 
there exists an element $a>0$ such that 
setting $h_{j}(z)=a(z-b_{j})^{d_{j}}+b_{j}$ ($j=2,\ldots ,m-1$), 
$G_{0}=\langle h_{1},\ldots ,h_{m-1}\rangle $ satisfies 
that $G_{0}\in {\cal G}$ and $\hat{K}(G_{0})=$ $K(h_{1})$ 
and if $\langle h_{1}\rangle $ is semi-hyperbolic (resp. hyperbolic), 
then $G_{0}$ is semi-hyperbolic (resp. hyperbolic). 
Combining it with Lemma~\ref{Constprop}, 
it follows that there exists an $a_{m}>0$ such that 
setting $h_{m}(z)=a_{m}(z-b_{m})^{d_{m}}+b_{m}$, 
$G=\langle h_{1},\ldots ,h_{m}\rangle $ satisfies that 
$G\in {\cal G}_{dis}$ and $\hat{K}(G)=\hat{K}(G_{0})=K(h_{1})$ and if 
$G_{0}$ is semi-hyperbolic (resp. hyperbolic), 
then $G$ is semi-hyperbolic (resp. hyperbolic).  

 Suppose now we have case (ii) and $d_{j}=2$ for each $j\geq 2.$  
 Then by Lemma~\ref{Constprop}, 
 there exists an $a_{2}>0$ such that 
setting $h_{j}(z)=a_{2}(z-b_{j})^{2}+b_{j}$ $(j=2,\ldots ,m)$, 
$G=\langle h_{1},\ldots ,h_{m}\rangle =\langle h_{1}, h_{2}\rangle $ 
satisfies that $G\in {\cal G}_{dis}$ and $\hat{K}(G)=K(h_{1})$ and if $\langle h_{1}\rangle $ 
is semi-hyperbolic (resp. hyperbolic), then 
$G$ is semi-hyperbolic (resp. hyperbolic). 

 Thus, we have proved Lemma~\ref{shshfinprop}.  
\end{proof} 
\begin{df}
Let $\Omega $ be the space of all non-empty compact subsets of 
 Poly$_{\deg \geq 2}$ endowed with the Hausdorff topology. 
We set 
\begin{itemize}
\item 
${\cal H}:= \{ \Gamma \in \Omega \mid \langle \Gamma \rangle \mbox{ is hyperbolic}\} ,$
\item 
${\cal B}:= \{ \Gamma \in \Omega \mid \langle \Gamma \rangle \in {\cal G}\} $, and 
\item 
${\cal D}:= \{ \Gamma \in \Omega \mid J(\langle \Gamma \rangle ) \mbox{ is disconnected}\} .$
\end{itemize} 
\end{df}
\begin{lem}
\label{l:omegahopen}
The sets ${\cal H}, {\cal H}\cap {\cal B}$, ${\cal H}\cap {\cal D}, {\cal H}\cap {\cal B}\cap {\cal D}$ 
are open in $\Omega .$ Moreover, $\Gamma \mapsto J(\langle \Gamma \rangle )$ is 
continuous on ${\cal H} $, with respect to the Hausdorff topology in the space of all non-empty compact 
subsets of $\CC .$ 
\end{lem}
\begin{proof}
We first show that ${\cal H}$ is open and $\Gamma \mapsto J(\langle \Gamma \rangle )$ is continuous on 
${\cal H}.$ 
In order to do that, 
let $\G \in {\cal H} .$ Then $P(\langle \Gamma \rangle )\subset F(\langle \Gamma \rangle ).$ 
Combining \cite[Theorem 2.14(5)]{S1} and 
Lemma~\ref{invnormal}, 
it follows that 
for each compact subset $K$ of $F(\langle \Gamma \rangle )$ and each 
neighborhood $U$ of $P(\langle \Gamma \rangle )$ in 
$F(\langle \Gamma \rangle )$, there exists an $n\in \NN $ such that 
for each $\gamma \in \GN $, 
$\gamma _{n}\cdots \gamma _{1}(K)\subset U.$ 
From this argument, ${\cal H}$ is open. 
Moreover, combining the above argument and Theorem~\ref{repdense}, 
it is easy to see that 
$\Gamma \mapsto J(\langle \Gamma \rangle )$ is continuous on ${\cal H}.$ 
 
Replacing $P(\langle \Gamma \rangle )$ by  $P^{\ast }(\langle \Gamma \rangle )$ 
and replacing  $F(\langle \Gamma \rangle )$ by int$(\hat{K}(\langle \Gamma \rangle ))$ 
in the above paragraph, we easily see that 
${\cal H}\cap {\cal B}$ is open.  
 
Since ${\cal H}$ is open and $\Gamma \mapsto J(\langle \Gamma \rangle )$ is continuous on ${\cal H}$, 
it is easy to see that ${\cal H}\cap {\cal D}$ is open. Therefore  
${\cal H}\cap {\cal B}\cap {\cal D}$ is open. 

Thus we have proved our lemma. 
\end{proof}
\begin{lem}
\label{l:twohypp}
Let $g_{1},g_{2}\in \mbox{{\em Poly}}_{\deg \geq 2}$ be hyperbolic. 
Suppose that $\langle g_{1}\rangle ,\langle g_{2}\rangle \in {\cal G}.$ 
Suppose also that $P^{\ast }(\langle g_{1}\rangle )\subset \mbox{{\em int}}(K(g_{2}))$ 
and $P^{\ast }(\langle g_{2}\rangle )\subset \mbox{{\em int}}(K(g_{1})).$ 
Then, there exists an $m\in \NN $ such that 
for each $n\in \NN $ with $n\geq m$, 
$\langle g_{1}^{n},g_{2}^{n}\rangle \in {\cal G}$ and $\langle g_{1}^{n},g_{2}^{n}\rangle $ 
is hyperbolic.  
\end{lem}
\begin{proof}
Let $U,V$ be two open neighborhood of 
$\bigcup _{i=1}^{2}P^{\ast }(\langle g_{i}\rangle )$ such that 
$\overline{V}\subset U\subset \overline{U}\subset  \bigcap _{i=1}^{2}\mbox{int}(K(g_{i})). $
Then there exists an $m\in \NN $ such that 
for each $n\in \NN $ with $n\geq m$, 
$\bigcup _{i=1}^{2}g_{i}^{n}(U)\subset V.$ 
It is easy to see that for each $n\in \NN $ with $n\geq m$, 
$P^{\ast }(\langle g_{1}^{n},g_{2}^{n}\rangle )\subset U\subset 
\mbox{int}(\hat{K}(\langle g_{1}^{n},g_{2}^{n}\rangle ))
\subset F(\langle g_{1}^{n},g_{2}^{n}\rangle ).$ 
Therefore for each $n\in \NN $ with $n\geq m$, 
$\langle g_{1}^{n}, g_{2}^{n}\rangle \in {\cal G}$ and 
$\langle g_{1}^{n}, g_{2}^{n}\rangle $ is hyperbolic. 
Thus we have proved our lemma. 
\end{proof} 
We give a sufficient condition so that statement~\ref{mainthran1-3} in Theorem~\ref{mainthran1} or 
statement~\ref{mainthran2-3} in Theorem~\ref{mainthran2} holds. 
\begin{lem}
\label{l:disjfinv}
Let $\Gamma $ be a compact subset of Poly$_{\deg \geq 2}$ and 
let $G=\langle \Gamma \rangle .$ 
Suppose that $G\in {\cal G}$, $G$ is semi-hyperbolic (resp. hyperbolic), and 
there exist two non-empty bounded open subsets $V_{1}, V_{2}$ of $\CC $ with $\overline{V_{1}}\cap \overline{V_{2}}=\emptyset $
 such that 
for each $i=1,2$, 
$\bigcup _{h\in \Gamma }h(V_{i})\subset V_{i}.$ 
Then, statement~\ref{mainthran1-3} in Theorem~\ref{mainthran1} (resp. 
statement~\ref{mainthran2-3} in Theorem~\ref{mainthran2}) holds.  
\end{lem}
\begin{proof}
Let $f:\Gamma ^{\NN }\times \CCI \rightarrow \Gamma ^{\NN }\times \CCI $ be the skew product associated with 
$\Gamma .$ 
By \cite[Theorem 2.14(1)]{S1}, for each $\gamma \in \GN $ and for 
each connected component $U$ of $F_{\gamma }(f)$, 
if $K$ is a compact subset of $U$, then diam$\gamma _{n}\cdots \gamma _{1}(K)\rightarrow 0$ 
as $n\rightarrow \infty .$ 
From our assumption, 
it follows that for each $\gamma \in \GN $, 
$F_{\gamma }(f)$ has at least two bounded components. 
Thus, the statement of our lemma holds. 
\end{proof}
\begin{rem}
\label{r:coexample}
Combining Lemma~\ref{Constprop}, \ref{shshfinprop}, 
\ref{l:omegahopen} and \ref{l:twohypp} (and their proofs), 
we easily obtain many examples of semi-hyperbolic (resp. hyperbolic) $G\in {\cal G}$ or $G\in {\cal G}_{dis}$,  
and we easily obtain many examples of $\Gamma $ such that 
statement~\ref{mainthran1-2} in Theorem~\ref{mainthran1} 
(resp. statement~\ref{mainthran2-2} in Theorem~\ref{mainthran2}) holds.  
 Moreover, combining Lemma~\ref{l:twohypp}, \ref{l:disjfinv} and their proofs, 
 we easily obtain many examples of $\Gamma $ such that 
 statement~\ref{mainthran1-3} in Theorem~\ref{mainthran1} or 
statement~\ref{mainthran2-3} in Theorem~\ref{mainthran2} holds. 
\end{rem}


\begin{thebibliography}{90}
\bibitem{Br} R. Br\"{u}ck,\ 
{\em Geometric properties of Julia sets of the composition of 
polynomials of the form $z^{2}+c_{n}$},\ 
Pacific J. Math.,\ {\bf 198} (2001), no. 2,\ 347--372.

\bibitem{BBR} R. Br\"{u}ck,\ M. B\"{u}ger and  
S. Reitz,\ {\em Random iterations of polynomials of the 
form $z^{2}+c_{n}$: Connectedness of Julia sets},\ 
Ergodic Theory Dynam. Systems,\ 
 {\bf 19},\ (1999),\ No.5,\ 1221--1231. 
\bibitem{Bu1} M. B\"{u}ger, 
{\em Self-similarity of Julia sets of the composition of 
polynomials}, 
Ergodic Theory  Dynam. Systems, {\bf 17} (1997), 1289--1297.
\bibitem{Bu2} M. B\"{u}ger, {\em 
On the composition of polynomials of the form} $z\sp 2+c\sb n$, 
 Math. Ann. 310 (1998), no. 4, 661--683.
\bibitem{CJY} L. Carleson, P. W. Jones and J. -C. Yoccoz, 
{\em Julia and John,\ }Bol. Soc. Brasil. Mat. (N.S.) {\bf 25,\ N.1} 1994,\ 1--30.
\bibitem{DH} L. DeMarco and S. L. Hruska, 
{\em Axiom A polynomial skew products of $\CC ^{2}$ and their 
postcritical sets}, Ergodic Theory  Dynam. Systems, 
(2008) {\bf 28}, 1749--1779. 
\bibitem{DH2} L. DeMarco and S. L. Hruska, 
{\em Correction to Axiom A polynomial skew products of $\CC ^{2}$ and their postcritical sets}, 
preprint, http://www.math.uic.edu/$\sim $demarco/correction.pdf. 
\bibitem{D} R. Devaney, {\em An Introduction to Chaotic 
Dynamical Systems}, Second edition, Perseus Books, Reading, Massachusetts, 1989. 
\bibitem{FS} J. E. Fornaess and  N. Sibony,\ 
{\em Random iterations of rational functions},\ 
Ergodic Theory Dynam. Systems, {\bf 11}(1991),\ 687--708.
\bibitem{GQL} Z. Gong, W. Qiu and Y. Li, 
{\em Connectedness of Julia sets for a quadratic random 
dynamical system}, Ergodic Theory Dynam. Systems, (2003), {\bf 23}, 
1807-1815.
\bibitem{GR} Z. Gong and F. Ren,\ {\em A random dynamical 
system formed by infinitely many functions,} 
 Journal of Fudan University, {\bf 35}, 1996,\ 387--392.

\bibitem{HM1} A. Hinkkanen and  G. J. Martin,
{\em The Dynamics of Semigroups of Rational Functions I},
 Proc. London Math. Soc. (3){\bf 73}(1996),\ 358--384.
\bibitem{HM2} A. Hinkkanen,\ G. J. Martin,
{\em Julia Sets of Rational Semigroups}, Math. Z. {\bf 222},\ 1996,\ no.2,\ 161--169.
\bibitem{J} M. Jonsson,\ 
{\em Ergodic properties of fibered rational maps},\ 
Ark. Mat.,\ 38 (2000), pp 281--317.
\bibitem{LV} O. Lehto and  K. I. Virtanen,\ 
{\em Quasiconformal Mappings in the plane,\ }
 Springer-Verlag,\ 1973.
\bibitem{M} J. Milnor, 
{\em Dynamics in One Complex Variable}
(Third Edition),  
Annals of Mathematical Studies, Number 160, 
Princeton University Press, 2006. 
\bibitem{NV} R. N\"{a}kki and  J. V\"{a}is\"{a}l\"{a},\ 
{\em John Discs,\ } Expo. math. 9(1991),\ 3--43.
\bibitem{PT} K. Pilgrim and Tan Lei,\ 
{\em Rational maps with disconnected Julia set},\ 
Asterisque\ 261,\ (2000), 349--384.
\bibitem{Se} O. Sester, 
{\em Hyperbolicit\'{e} des polyn\^{o}mes fibr\'{e}s}, 
Bull. Soc. Math. France {\bf 127} 1999, no. 3, 393--428. 
\bibitem{Sta1} R. Stankewitz,\ 
{\em Completely invariant Julia sets of polynomial semigroups,\ }
Proc. Amer. Math. Soc.,\ {\bf 127},\ (1999),\ No. 10,\ 2889--2898. 
\bibitem{Sta2} R. Stankewitz,\ 
{\em Completely invariant sets of normality for rational 
semigroups},\ Complex Variables Theory Appl.,\ Vol. 40.(2000),\ 199--210.
\bibitem{Sta3} R. Stankewitz,\ {\em Uniformly 
perfect sets,\ rational 
semigroups,\ Kleinian groups and IFS's ,\ }Proc. Amer. Math. Soc.
{\bf 128},\ (2000),\ No. 9,\ 2569--2575.
\bibitem{SSS} R. Stankewitz, T. Sugawa, and H. Sumi,\ 
{\em Some counterexamples in dynamics of rational semigroups},\ 
Annales Academiae Scientiarum Fennicae Mathematica Vol. 29,\ 
2004,\ 357--366.
\bibitem{SS} R. Stankewitz and H. Sumi, 
{\em Dynamical properties and structure of Julia sets of 
postcritically bounded polynomial semigroups},  
to appear in Trans. Amer. Math. Soc., http://arxiv.org/abs/0708.3187.
\bibitem{Ste} N. Steinmetz, 
{\em Rational Iteration}, 
de Gruyter Studies in Mathematics 16,  
Walter de Gruyter, 1993.
\bibitem{Steins} D. Steinsaltz, 
{\em Random logistic maps Lyapunov exponents}, 
Indag. Mathem., N. S., 12 (4), 557--584, 2001. 
\bibitem{S1} H. Sumi,\ {\em Dynamics of sub-hyperbolic and 
semi-hyperbolic rational semigroups and skew products},\  
Ergodic Theory Dynam. Systems, (2001),\ {\bf 21},\ 563--603.
\bibitem{S2} H. Sumi,\ 
{\em A correction to the proof of a lemma in 
`Dynamics of sub-hyperbolic and semi-hyperbolic 
rational semigroups and skew products'},\  
Ergodic Theory Dynam. Systems, (2001),\ {\bf 21},\ 1275--1276.

\bibitem{S3} H. Sumi,\ {\em Skew product maps related to finitely 
generated rational semigroups,\ } 
Nonlinearity,\ {\bf 13},\ (2000), 
995--1019.

\bibitem{S4} H. Sumi,\ 
{\em Semi-hyperbolic fibered rational maps and rational 
semigroups},\ 
Ergodic Theory Dynam. Systems, (2006), 
{\bf 26}, 893--922. 
\bibitem{S4c} H. Sumi, 
{\em Erratum to: ``Semi-hyperbolic fibered rational maps and rational semigroups'' [Ergodic Theory Dynam. Systems 26 (2006), no. 3, 893--922].} Ergodic Theory Dynam. Systems 28 (2008), no. 3, 1043--1045.
\bibitem{S5} H. Sumi,\ {\em On Dynamics of 
Hyperbolic Rational Semigroups,}\ 
Journal of Mathematics of Kyoto University,\ Vol. 37,\ 
No. 4,\ 1997,\ 717--733.
\bibitem{S6} H. Sumi,\ 
{\em Dimensions of Julia sets of expanding rational semigroups},\ 
Kodai Mathematical Journal,\ Vol. 28,\ No. 2,\ 2005,\ pp390--422. 
(See also http://arxiv.org/abs/math.DS/0405522.)
\bibitem{S9} H. Sumi, 
{\em Random dynamics of polynomials and devil's-staircase-like 
functions in the complex plane}, 
Applied Mathematics and Computation 187 (2007) pp489-500. 
(Proceedings paper of a conference. ) 
\bibitem{S10} H. Sumi, 
{\em The space of postcritically bounded 2-generator polynomial 
semigroups with hyperbolicity}, 
RIMS Kokyuroku 1494, 62--86, 2006. 
(Proceedings paper of a conference.) 
\bibitem{S7} H. Sumi,
{\em Interaction cohomology of forward or backward self-similar systems}, 
Adv. Math., 222 (2009), no. 3,  729--781.  
\bibitem{SdpbpI} 
H. Sumi, {\em Dynamics of postcritically bounded polynomial semigroups I: 
connected components of the Julia sets}, 
preprint 2008, 
http://arxiv.org/abs/0811.3664.
\bibitem{S11} H. Sumi, 
{\em Dynamics of postcritically bounded polynomial semigroups II: 
fiberwise dynamics and the Julia sets}, 
preprint 2008. 
\bibitem{S8} H. Sumi, 
{\em Random complex dynamics and semigroups of holomorphic maps}, 
preprint 2008, 
http://arxiv.org/abs/0812.4483.  
\bibitem{Snew} H. Sumi, 
in preparation.
 
\bibitem{SU2} H. Sumi and M. Urba\'{n}ski,\ 
{\em The equilibrium states for semigroups of rational 
maps}, Monatsh. Math., 156 (2009), no. 4, 371--390. 

\bibitem{SU1} H. Sumi and M. Urba\'{n}ski,\ 
{\em Real analyticity of Hausdorff dimension for 
expanding rational semigroups}, 
to appear in Ergodic Theory Dynam. Systems,  
http://arxiv.org/abs/0707.2447.
\bibitem{SU3} H. Sumi and M. Urba\'{n}ski, 
{\em Measures and dimensions of Julia sets of semi-hyperbolic rational semigroups}, 
preprint 2008, http://arxiv.org/abs/0811.1809.  
\bibitem{SY} Y. Sun and C-C. Yang.\ 
{\em On the connectivity of the Julia set 
of a finitely generated rational semigroup},\ 
Proc. Amer. Math. Soc.,\ Vol. 130,\ 
No. 1,\ 49--52,\ 2001.

\bibitem{ZR} W. Zhou,\ F. Ren,\ 
{The Julia sets of the random iteration of rational functions, } 
Chinese Sci. Bulletin, {\bf 37}(12),\ 1992,\ 969--971.
\end{thebibliography}
\end{document}